\documentclass[12pt]{amsart}
\usepackage[utf8]{inputenc}
\usepackage[shortlabels]{enumitem}
\usepackage{amssymb}
\usepackage{mathrsfs}
\usepackage[all]{xy}
\usepackage[margin=1in]{geometry} 
\usepackage[bookmarks, bookmarksdepth=2, colorlinks=true, linkcolor=blue, citecolor=blue, urlcolor=blue]{hyperref}
\usepackage{eucal}
\usepackage{contour}
\usepackage[normalem]{ulem}
\usepackage{xcolor}

\numberwithin{equation}{section}
\newtheorem{theorem}[equation]{Theorem}

\newtheorem{proposition}[equation]{Proposition}
\newtheorem{lemma}[equation]{Lemma}
\newtheorem{corollary}[equation]{Corollary}

\theoremstyle{definition}
\newtheorem{rmk}[equation]{Remark}
\newenvironment{remark}[1][]{\begin{rmk}[#1] \pushQED{\qed}}{\popQED \end{rmk}}
\newtheorem{eg}[equation]{Example}
\newenvironment{example}[1][]{\begin{eg}[#1] \pushQED{\qed}}{\popQED \end{eg}}
\newtheorem{defnaux}[equation]{Definition}
\newenvironment{definition}[1][]{\begin{defnaux}[#1]\pushQED{\qed}}{\popQED \end{defnaux}}

\newcommand{\cA}{\mathcal{A}}

\newcommand{\rA}{\mathrm{A}}

\newcommand{\cB}{\mathcal{B}}

\newcommand{\rB}{\mathrm{B}}

\newcommand{\bC}{\mathbf{C}}
\newcommand{\cC}{\mathcal{C}}

\newcommand{\sE}{\mathscr{E}}
\newcommand{\bF}{\mathbf{F}}
\newcommand{\cF}{\mathcal{F}}

\newcommand{\cP}{\mathcal{P}}

\newcommand{\bQ}{\mathbf{Q}}

\newcommand{\bS}{\mathbf{S}}

\newcommand{\fS}{\mathfrak{S}}

\newcommand{\bT}{\mathbf{T}}

\newcommand{\cU}{\mathcal{U}}

\newcommand{\sY}{\mathscr{Y}}
\newcommand{\bZ}{\mathbf{Z}}

\newcommand{\nlet}{%
  \mathrel{\ooalign{$\le$\cr\hidewidth$|$\hidewidth}}%
}

\newcommand{\arxiv}[1]{\href{http://arxiv.org/abs/#1}{{\tiny\tt arXiv:#1}}}
\newcommand{\DOI}[1]{\href{http://doi.org/#1}{\color{purple}{\tiny\tt DOI:#1}}}
\newcommand{\stacks}[1]{\cite[\href{http://stacks.math.columbia.edu/tag/#1}{Tag~#1}]{stacks}}
\newcommand{\defn}[1]{\emph{#1}}
\renewcommand{\AA}{\mathbb{A}}
\newcommand{\BB}{\mathbb{B}}
\renewcommand{\phi}{\varphi}
\renewcommand{\emptyset}{\varnothing}
\DeclareMathOperator{\im}{im} 
\DeclareMathOperator{\Sym}{Sym}
\DeclareMathOperator{\Aut}{Aut}
\DeclareMathOperator{\Hom}{Hom}
\DeclareMathOperator{\Rep}{Rep}
\newcommand{\id}{\mathrm{id}}
\newcommand{\op}{\mathrm{op}}

\DeclareMathOperator{\age}{age}
\DeclareMathOperator{\Perm}{Perm}

\DeclareMathOperator{\Eq}{Eq}
\newcommand{\orb}{\mathrm{orb}}
\newcommand{\bzero}{\mathbf{0}}
\newcommand{\bone}{\mathbf{1}}

\DeclareMathOperator{\Amalg}{Amalg}
\newcommand{\FinSet}{\mathbf{FinSet}}

\contourlength{1pt}

\contourlength{0.8pt}
\newcommand{\myuline}[1]{%
  \uline{\phantom{#1}}%
  \llap{\contour{white}{#1}}%
}
\DeclareMathOperator{\uRep}{\text{\myuline{\rm Rep}}}

\title{Pre-Galois categories and Fra\"iss\'e's theorem}

\author{Nate Harman}
\author{Andrew Snowden}

\date{December 7, 2023}

\begin{document}

\begin{abstract}
Galois categories can be viewed as the combinatorial analog of Tannakian categories. We introduce the notion of \defn{pre-Galois category}, which can be viewed as the combinatorial analog of pre-Tannakian categories. Given an oligomorphic group $G$, the category $\bS(G)$ of finitary smooth $G$-sets is pre-Galois. Our main theorem (approximately) says that these examples are exhaustive; this result is, in a sense, a reformulation of Fra\"iss\'e's theorem. We also introduce a more general class of \defn{$\rB$-categories}, and give some examples of $\rB$-categories that are not pre-Galois using permutation classes. This work is motivated by certain applications to pre-Tannakian categories.
\end{abstract}

\maketitle
\tableofcontents

\section{Introduction}

\subsection{Background}

The famous Tannakian reconstruction theorem says that an algebraic group can be recovered from its representation category. To be a bit more precise, fix an algebraically closed field $k$. A \defn{pre-Tannakian category} is a $k$-linear abelian category equipped with a symmetric tensor structure satisfying some axioms. A \defn{Tannakian category} is a pre-Tannakian category $\cC$ equipped with a \defn{fiber functor} $\omega$, i.e., a faithful exact tensor functor to finite dimensional vector spaces. The motivating example of a Tannakian category is the category $\Rep_k(G)$ of finite dimensional representations of an algebraic group $G/k$; the fiber functor is simply the forgetful functor. The main theorem of Tannakian categories states these examples are essentially exhaustive: if $\cC$ is a Tannakian category then $\cC$ is equivalent to $\Rep_k(G)$, where $G$ is the (pro-algebraic) automorphism group of $\omega$. See \cite{DeligneMilne} for details.

It is not so easy to construct pre-Tannakian categories that are not (super-)Tannakian, but a number of interesting examples are known, including Deligne's interpolation categories \cite{Deligne3}, the Delannoy category \cite{line}, and the Verlinde category \cite{Ostrik}. A major problem in the field of tensor categories is to better understand the pre-Tannakian landscape.

There is a combinatorial analog of Tannakian reconstruction, in the form of Grothendieck's Galois theory. A \defn{Galois category} is a category $\cC$ equipped with a functor $\omega$ to finite sets satisfying certain axioms (see Definition~\ref{def:galois}). The motivating example of a Galois category is the category of finite $G$-sets, for a group $G$. The main theorem of Galois categories states that these examples are essentially exhaustive: if $\cC$ is a Galois category then $\cC$ is equivalent to the category of smooth (=discrete) $G$-sets, where $G$ is the (profinite) automorphism group of $\omega$. Grothendieck applied this theorem to construct the \'etale fundamental group.

Conspicuously absent from the combinatorial side is an analog of pre-Tannakian categories. The purpose of this paper is to fill this gap: we define this class of categories, prove one main theorem about them, and construct some interesting examples.

\subsection{Pre-Galois categories}

The following definition introduces our combinatorial analog of pre-Tannakian categories:

\begin{definition} \label{defn:galois-intro}
An essentially small category $\cB$ is \defn{pre-Galois} if the following conditions hold:
\begin{enumerate}
\item $\cB$ has finite co-products (and thus an initial object $\bzero$).
\item Every object of $\cB$ is isomorphic to a finite co-product of \defn{atoms}, i.e., objects that do not decompose under co-product.
\item If $X$ is an atom and $Y$ and $Z$ are other objects, then any map $X \to Y \amalg Z$ factors uniquely through $Y$ or $Z$.
\item $\cB$ has fiber products and a final object $\bone$.
\item Any monomorphism of atoms is an isomorphism.
\item If $X \to Z$ and $Y \to Z$ are maps of atoms then $X \times_Z Y$ is non-empty (i.e., not $\bzero$).
\item The final object $\bone$ is atomic.
\item Equivalence relations in $\cB$ are effective (see Definition~\ref{defn:eff-eq}). \qedhere
\end{enumerate}
\end{definition}

The above axioms are motivated by properties of the category of finite $G$-sets, for a group $G$. ``Atoms'' should be thought of as transitive $G$-sets. The first three axioms basically say that objects admit a finite ``orbit decomposition'' which behaves in the expected manner.

We define a \defn{$\rB$-category} to be one satisfying axioms (a)--(e). This turns out to be a very nice class of categories already. For example, we show that every $\rB$-category is balanced (Corollary~\ref{cor:balanced}) and has finite Hom sets (Proposition~\ref{prop:Hom-finite}). Axiom (e) is somewhat subtle, but these nice properties of $\rB$-categories depend on it.

Of the remaining three axioms, (f) is clearly the most important: in a sense, it is easy to explain all failures of (g) and (h), but this is not the case for (f). We say that a $\rB$-category is \defn{non-degenerate} if it satisfies (f) and (g). Non-degeneracy implies a number of nice properties, such as existence of co-equalizers. Axiom (h) ensures that quotients are well-behaved.

One can match properties of pre-Galois categories and pre-Tannakian categories, to some extent. Axiom (a) corresponds to additivity on the pre-Tannakian side. Both pre-Galois and pre-Tannakian categories are finitely complete and co-complete. Axiom (h) corresponds to the first isomorphism theorem on the pre-Tannakian side.  Axiom (b) corresponds to the finite length condition on the pre-Tannakian side. The co-product and product in a pre-Galois category correspond to the direct sum and tensor product in a pre-Tannakian category. Axiom (g) corresponds to the pre-Tannakian axiom that the unit object is simple. Finally, (f) corresponds to the fact that in a pre-Tannkain category the tensor product of non-zero objects is non-zero.

\subsection{Examples}

For any group $G$, the category of finite $G$-sets is a pre-Galois category, and this is the motivating example. One might try to construct other examples by considering (possibly infinite) $G$-sets with finitely many orbits. This does not work in general since a product of two such $G$-sets need not have finitely many orbits. For instance, $G$ acting on itself by left multiplication has one orbit, but the orbits of $G$ on $G \times G$ are in bijection with $G$ itself.

It turns out that the above idea can be made to work in at least one situation, however. Recall that an \defn{oligomorphic group} is a permutation group $(G, \Omega)$ such that $G$ has finitely many orbits on $\Omega^n$ for all $n \ge 0$. The simplest example of an oligomorphic group is the infinite symmetric group. Model theory, and the theory of Fra\"iss\'e limits in particular, provides many more examples. See \cite{CameronBook} for general background. Given an oligomorphic group $G$, we define $\bS(G)$ to be the category of sets equipped with an action of $G$ that is smooth (every stabilizer is open in the natural topology) and which has finitely many orbits. It turns out that this category \emph{is} closed under products; this is a consequence of the oligomorphic condition. It is not hard to show that $\bS(G)$ is in fact pre-Galois.

The above examples admit a mild generalization: we define a class of topological groups called \defn{admissible groups}, which include profinite groups and oligomorphic groups, and we associate a pre-Galois category $\bS(G)$ to such $G$. From the topological perspective, the key finiteness property of $G$ is Roelcke pre-compactness. We review this theory in \S \ref{s:olig}; a more detailed treatment can be found in \cite[\S 2]{repst}.

In Example~\ref{ex:separable} we give a non-trivial example of a degenerate $\rB$-category using a non-Fra\"iss\'e class of relational structures. It would be interesting if one could give a more direct construction of such an example.

\subsection{The main theorem}

The following is our main result on pre-Galois categories.

\begin{theorem}[Theorem~\ref{thm:uncountable}] \label{mainthm}
If $\cB$ is a pre-Galois category then $\cB$ is equivalent to $\bS(G)$ for some admissible group $G$.
\end{theorem}

In fact, we prove a bit more than this: Theorem~\ref{thm:fraisseB} accommodates all (countable) non-degenerate $\rB$-categories; in other words, we still obtain a classification result when we do not impose Definition~\ref{defn:galois-intro}(h). The non-degeneracy condition seems essential, however.

\begin{remark}
The group $G$ in Theorem~\ref{mainthm} is not unique. There are essentially two reasons for this:
\begin{enumerate}
\item One can always replace $G$ with a dense subgroup without changing the category $\bS(G)$. For example, if $G$ is the infinite symmetric group (all permutations of $\{1,2,\ldots\}$) and $G'$ is the finitary version (the union of the $\fS_n$) then $\bS(G) \cong \bS(G')$.
\item There are admissible groups of different ``sizes'' that are essentially the same, and have the same $\bS$ category. For example, if $G$ and $G'$ are the groups of all permutations on two infinite sets of different cardinalities then $\bS(G) \cong \bS(G')$.
\end{enumerate}
It is possible to eliminate the first issue by working with complete groups. See \cite[\S 7]{pro-etale} for a discussion of completeness. In general, there does not seem to be a good way to deal with the second issue. However, if $\cB$ has countably many isomorphism classes then one can take $G$ to be first-countable, and this does deal with the issue: in this case, there is a unique (up to isomorphism) complete first-countable $G$.
\end{remark}

\subsection{Overview of proof}

Let $\cB$ be a $\rB$-category, and let $\cA$ be the full subcategory of $\cB^{\op}$ spanned by atoms. We show that $\cB$ can be recovered from $\cA$, and exactly characterize the categories $\cA$ that arise in this manner (we call them $\rA$-categories). The key point in the proof of Theorem~\ref{mainthm} is that $\cA$ is a \defn{Fra\"iss\'e category}, meaning it is the kind of category to which the categorical version of Fra\"iss\'e's theorem applies. This theorem produces a universal homogeneous ind-object $\Omega$ in $\cA$. We show that $G=\Aut(\Omega)$ is naturally an admissible group, and that $\cB$ is equivalent to $\bS(G)$.

Actually, there is a technical detail elided above: we only show that $\cA$ is a Fra\"iss\'e category under the condition that it has countably many isomorphism classes. We first prove the theorem under this hypothesis, and then deduce the general case from this special case.

The correspondence between $\rA$- and $\rB$-categories is also useful for producing examples of $\rB$-categories: indeed, it is easy to construct $\rA$-categories by taking classes of relational structures, and one can then convert them to $\rB$-categories. We follow this plan in \S \ref{s:examples}.

\subsection{Motivation}

We now say a few words about why we wanted to develop this theory.

(a) In a recent paper \cite{repst}, we constructed a pre-Tannakian category $\uRep_k(G,\mu)$ associated to an oligomorphic group $G$ equipped with a measure $\mu$ (in a sense that we introduced), satisfying certain conditions. Our construction recovers Deligne's interpolation categories in certain cases, and leads to new categories (like the Delannoy category) in other cases. Some constructions and results in \cite{repst} hold for more general $\rB$-categories, and this was our original motivation for developing the theory.

(b) One might expect pre-Galois categories to be easier to understand than pre-Tannakian categories, and yet behave similarly in some ways. Here is one example of such a parallel. Deligne \cite{Deligne2} proved that a pre-Tannakian category (in characteristic~0) is Tannakian if and only if it is of moderate growth, meaning the length of $X^{\otimes n}$ grows at most exponentially with $n$ for each object $X$. This is an important theorem since it intrinsically characterizes Tannakian categories. Using our Theorem~\ref{mainthm}, one can prove an analogous result here: a pre-Galois category is Galois if and only if it is of moderate growth, meaning $X^n$ has at most exponentially summands for each $X$.

(c) Since oligomorphic groups are rather foreign to the field of tensor categories, our constructions from \cite{repst} may seem to come out of left field. This paper shows that the connection is in fact inevitable: pre-Galois categories are a natural combinatorial analog of pre-Tannakian categories, and they are intrinsically linked to oligomorphic groups. In forthcoming work \cite{frobalg}, we draw an even more direct connection: using Theorem~\ref{mainthm}, we show that every discrete pre-Tannakian category comes from an oligomorphic group. (``Discrete'' means that the category is generated by its \'etale algebras.)

%

\subsection{Outline}

In \S \ref{s:olig}, we review oligomorphic and admissible groups and the associated categories $\bS(G)$; these are the motivating examples of pre-Galois categories. In \S \ref{s:Bcat}, we define $\rB$-categories and establish some of their basic properties. In \S \ref{s:pre-galois}, we introduce pre-Galois categories, and establish some of their special features. In \S \ref{s:Acat}, we study the category of atoms in a $\rB$-category, which leads to the notion of $\rA$-category. In \S \ref{s:fraisse}, we review Fra\"iss\'e theory and prove our main theorem. Finally, in \S \ref{s:examples}, we give some examples of $\rA$- and $\rB$-categories coming from relational structures.

\subsection{Notation}

We list some of the important notation here:
\begin{description}[align=right,labelwidth=2.5cm,leftmargin=!]
\item[ $\bzero$ ] the initial object of a $\rB$-category (e.g., the empty set)
\item[ $\bone$ ] the final object of a $\rB$-category (e.g., the one-point set)
\item[ $\bS(G)$ ] the category of finitary (and smooth) $G$-sets
\item[ $\bT(G)$ ] the category of transitive (and smooth) $G$-sets
\item[ $\AA(\cB)$ ] the A-category associated to $\cB$ (see \S \ref{ss:AA})
\item[ $\BB(\cA)$ ] the B-category associated to $\cA$ (see \S \ref{ss:AA})
\end{description}

\section{Oligomorphic groups} \label{s:olig}

In this section, we review oligomorphic and admissible groups, and recall the category $\bS(G)$ of finitary $G$-sets. These categories are the motivation for the general notion of pre-Galois category we study in this paper.

\subsection{Oligomorphic groups}

An \defn{oligomorphic group} is a permutation group $(G, \Omega)$ such that $G$ has finitely many orbits on $\Omega^n$ for all $n \ge 0$. Here are a few concrete examples:
\begin{itemize}
\item The infinite symmetric group $\fS$, i.e., the group of all permutations of $\Omega=\{1,2,\ldots\}$.
\item The infinite general linear group over a finite field $\bF$, i.e., the group of all linear automorphisms of $\bF^{\oplus \infty}$.
\item The group of all order-preserving self-bijections of $\bQ$.
\end{itemize}
Many more examples can be obtained from Fra\"iss\'e limits. For example, if $R$ is the Rado graph (which is the Fra\"iss\'e limit of all finite graphs) then $\Aut(R)$ acts oligomorphically on the vertex set of $R$. We refer to Cameron's book \cite{CameronBook} for general background on oligomorphic groups.

\subsection{Admissible groups}

Fix an oligomorphic group $(G, \Omega)$. For a finite subset $A$ of $\Omega$, let $G(A)$ be the subgroup of $G$ fixing each element of $A$. The groups $G(A)$ form a neighborhood basis of the identity for a topology on $G$. This topology has the following properties \cite[\S 2.2]{repst}:
\begin{itemize}
\item It is Hausdorff.
\item It is non-archimedean: open subgroups form a neighborhood basis of the identity.
\item It is Roelcke pre-compact: if $U$ and $V$ are open subgroups then $V \backslash G/U$ is finite.
\end{itemize}
We define an \defn{admissible group} to be a topological group with these three properties. Thus every oligomorphic group gives rise to an admissible group. We also note that any finite group is admissible (with the discrete topology), and any profinite group is admissible. While we are most interested in oligomorphic groups, we typically will not have a preferred permutation action, and so it is most natural to work with admissible groups.

\subsection{Actions}

Let $G$ be an admissible group. We say that an action of $G$ on a set $X$ is \defn{smooth} if all stabilizers are open. We use the term ``$G$-set'' to mean ``set equipped with a smooth action of $G$.'' We say that a $G$-set is \defn{finitary} if it has finitely many orbits. We write $\bS(G)$ for the category of finitary $G$-sets (with morphisms being $G$-equivariant maps), and $\bT(G)$ for the full subcategory on the transitive $G$-sets. An important property of $\bS(G)$ is that it is closed under products and fiber products; see \cite[\S 2.3]{repst}.

%


There is a variant of the category $\bS(G)$ that will play an important role. A \defn{stabilizer class} in $G$ is a collection $\sE$ of open subgroups of $G$ satisfying the following conditions:
\begin{enumerate}
\item $\sE$ contains $G$.
\item $\sE$ is closed under conjugation.
\item $\sE$ is closed under finite intersections.
\item $\sE$ forms a neighborhood basis for the identity of $G$.
\end{enumerate}
We say that a $G$-set is \defn{$\sE$-smooth} if its stabilizers all belong to $\sE$. We let $\bS(G; \sE)$ be the full subcategory of $\bS(G)$ spanned by the $\sE$-smooth sets, and analogously define $\bT(G; \sE)$. The category $\bS(G; \sE)$ is also closed under products and fiber products.

\begin{example} \label{ex:sym-stab-class}
Let $\fS$ be the infinite symmetric group, let $\fS(n) \subset \fS$ be the subgroup fixing each of $1, \ldots, n$, and let $\fS_n$ be the symmetric group on $n$ letters. Let $\sE$ be the set of all subgroups of $\fS$ conjugate to some $\fS(n)$, and let $\sY$ be the set of all subgroups of $\fS$ conjugate to one of the form $\fS_{m_1} \times \cdots \fS_{m_r} \times \fS(n)$, where $m_1+\cdots+m_r=n$. Then $\sE$ and $\sY$ are stabilizer classes in $\fS$.
\end{example}

%
%
%
%
%

\section{Combinatorial tensor categories} \label{s:Bcat}

In this section, we introduce the class of $\rB$-categories, which we view as combinatorial analogs of tensor categories. All categories in this section are essentially small.

\subsection{Basic definitions}

Let $\cB$ be a category with finite co-products. We write $\bzero$ for the initial object and refer to it (or any object isomorphic to it) as \defn{empty}. We say that an object $X$ is \defn{atomic}, or an \defn{atom}, if it is non-empty and does not decompose non-trivially under co-product; that is, given an isomorphism $X \cong Y \amalg Z$ either $Y$ or $Z$ is empty.

We now introduce our combinatorial analog of tensor categories.

\begin{definition} \label{defn:B}
A \defn{$\rB$-category} is an essentially small category $\cB$ satisfying the following conditions:
\begin{enumerate}
\item $\cB$ has finite co-products.
\item Every object of $\cB$ is isomorphic to a finite co-product of atoms.
\item Given objects $X$, $Y$, and $Z$, with $X$ atomic, the natural map
\begin{displaymath}
\Hom(X, Y) \amalg \Hom(X, Z) \to \Hom(X, Y \amalg Z)
\end{displaymath}
is a bijection.
\item $\cB$ has fiber products and a final object $\bone$.
\item Any monomorphism of atoms is an isomorphism.
\end{enumerate}
We also define a \defn{$\rB_0$-category} to be an essentially small category satisfying (a)--(c), and a \defn{$\rB_1$-category} to be one satisfying (a)--(d).
\end{definition}

The following proposition establishes the motivating example.

\begin{proposition}
Let $G$ be an admissible group and let $\sE$ be a stabilizer class. Then the category $\bS(G; \sE)$ is a $\rB$-category.
\end{proposition}

\begin{proof}
(a) The co-product is given by disjoint union.

(b) Atoms are transitive $\sE$-smooth $G$-sets. Every finitary $\sE$-smooth $G$-set is clearly a finite disjoint union of transitive $\sE$-smooth $G$-sets.

(c) Suppose $X$ is an atom and $Y$ and $Z$ are arbitrary objects of $\bS(G; \sE)$. Let $f \colon X \to Y \amalg Z$ be a map. If any point of $X$ maps into $Y$ (or $Z$) then all of $X$ maps into $Y$ (or $Z$) since the map is $G$-equivariant and $G$ acts transitively on $X$. Thus axiom (c) holds.

(d) The ordinary fiber product of sets is the fiber product in $\bS(G; \sE)$. The final object is the one-point $G$-set (which is $\sE$-smooth since $\sE$ is required to contain $G$).

(e) Suppose $f \colon X \to Y$ is a monomorphism of atoms in $\bS(G; \sE)$. As in any category with fiber products, this implies that the projection map $X \times_Y X \to X$ is an isomorphism. Since the set underlying $X \times_Y X$ is just the usual fiber product of sets, we see that $f$ is an injective function. Since $f$ is an injective map of transitive $G$-sets, it is bijective, and thus an isomorphism in the category.
\end{proof}

\begin{remark}
We mention a few simple ways of producing new $\rB$-categories.
\begin{enumerate}
\item Let $\cB$ be a B-category and let $X$ be an object of $B$. Let $\Sigma$ be the class of all atomic objects appearing as a summand of $X^n$ for some $n$. Let $\cB'$ be the full subcategory of $\cB$ spanned by objects that are co-products of objects in $\Sigma$. Then $\cB'$ is a B-category; we call this the subcategory \defn{generated} by $X$.
\item Let $\cB$ be a B-category and let $S$ be an object of $\cB$. Then the category $\cB_{/S}$ of objects over $S$ is a B-category. If $\cB=\bS(G)$ and $S=G/U$ for an open subgroup $U$ then $\cB_{/S}=\bS(U)$.
\item Suppose $\cB_1$ and $\cB_2$ are B-categories. Then the product category $\cB_1 \boxplus \cB_2$ is a B-category; we call it the \defn{sum} category.
\item Let $\cB$ be a $\rB$-category and let $\bone=S_1 \amalg \cdots \amalg S_n$ be the atomic decomposition of the final object. Then $\cB$ is naturally equivalent to $\cB_{/S_1} \boxplus \cdots \boxplus \cB_{/S_n}$, and each $\cB_{/S_i}$ has an atomic final object. \qedhere
\end{enumerate}
\end{remark}

\subsection{Properties of $\rB_0$-categories} \label{ss:B0-prop}

Although we are mostly interested in $\rB$-categories, some results hold in greater generality, and this additional generality is useful in later proofs. In this spirit, we now prove some basic results about $\rB_0$-categories. We fix a $\rB_0$-category $\cB$ for \S \ref{ss:B0-prop}.

\begin{proposition} \label{prop:map-to-empty}
If $X$ is non-empty then there are no maps $X \to \bzero$.
\end{proposition}

\begin{proof}
It suffices to treat where $X$ is atomic, so we assume this. By Definition~\ref{defn:B}(c) the natural map
\begin{displaymath}
\Hom(X, \bzero) \amalg \Hom(X, \bzero) \to \Hom(X, \bzero \amalg \bzero)=\Hom(X, \bzero)
\end{displaymath}
is bijective, and so $\Hom(X, \bzero)=\emptyset$ as required.
\end{proof}

\begin{proposition} \label{prop:Bmap}
Let $f \colon X \to Y$ be a morphism. Write $X=X_1 \amalg \cdots \amalg X_n$ and $Y=Y_1 \amalg \cdots \amalg Y_m$ where each $X_i$ and $Y_i$ is atomic. There exists a unique function $a \colon [n] \to [m]$ such that the restriction of $f$ to $X_i$ factors uniquely through $Y_{a(i)}$; let $f_i \colon X_i \to Y_{a(i)}$ be the induced map. Then $f$ is uniquely determined by $a$ and the $f_i$'s. Moreover, every choice of $a$ and the $f_i$'s comes from some $f$.
\end{proposition}

\begin{proof}
For each $i$, the natural map
\begin{displaymath}
\coprod_{j=1}^m \Hom(X_i, Y_j) \to \Hom(X_i, Y)
\end{displaymath}
is a bijection. For $m=0$, this is Proposition~\ref{prop:map-to-empty}, for $m=1$ it is obvious, and for $m \ge 2$ it follows from Definition~\ref{defn:B}(c) inductively. We thus see that, given $i$, there is a unique $a(i) \in [m]$ and a unique morphism $f_i \colon X_i \to Y_{a(j)}$ such that the restriction of $f$ to $X_i$ is $f_i$ following by the natural map $Y_{a(j)} \to Y$. This proves the existence of $a$ and the $f_i$'s. That they determine $f$, and that every choice arises, follows from the definition of co-product.
\end{proof}

\begin{corollary} \label{cor:Bmap}
Let $f, f' \colon X \to X'$ and $g, g' \colon Y \to Y'$ be morphisms. Then $f \amalg g=f' \amalg g'$ if and only if $f=f'$ and $g=g'$.
\end{corollary}

\begin{proposition} \label{prop:monic-sum}
Let $f \colon X \to X'$ and $g \colon Y \to Y'$ be morphisms. Then:
\begin{enumerate}
\item $f \amalg g$ is monomorphic if and only if $f$ and $g$ are monomorphic.
\item $f \amalg g$ is epimorphic if and only if $f$ and $g$ are epimorphic.
\end{enumerate}
\end{proposition}

\begin{proof}
(a) First suppose that $f$ is not monomorphic. Let $h,h' \colon W \to X$ be distinct maps such that $fh=fh'$.  Then $h \amalg \id_Y$ and $h' \amalg \id_Y$ are maps $W \amalg Y \to X \amalg Y$, which are distinct by Corollary~\ref{cor:Bmap}, but have the same composition with $f \amalg g$. Thus $f \amalg g$ is not monomorphic.

Now suppose that $f$ and $g$ are monomorphic. Let $h,h' \colon W \to X \amalg Y$ be maps that have equal composition with $f \amalg g$. We show $h=h'$. It suffices to treat the case where $W$ is atomic, since a map out of $W$ is determined by its restrictions to the summands of $W$. Thus assume $W$ is atomic. Then $W$ maps into exactly one of $X$ or $Y$ under $h$; without loss of generality, say $X$. Then $W$ maps into $X'$ under $(f \amalg g) \circ h$. It follows that $W$ also maps into $X'$ under $(f \amalg g) \circ h'$, and so must map into $X$ under $h'$. Regarding $h$ and $h'$ as maps into $X$, we thus have $(fh \amalg g)=(fh' \amalg g)$ as maps $W \amalg Y \to W \amalg Y'$, and so $fh=fh'$ by Corollary~\ref{cor:Bmap}. Since $f$ is monomorphic, we conclude $h=h'$. Thus $f \amalg g$ is monomorphic.

(b) First suppose that $f$ is not epimorphic. Let $h,h' \colon X' \to Z$ be distinct maps such that $hf=h'f$. Then $h \amalg \id_{Y'}$ and $h' \amalg \id_{Y'}$ are maps $X' \amalg Y' \to X \amalg Y'$, which are distinct by Corollary~\ref{cor:Bmap}, but have the same composition with $f \amalg g$. Thus $f \amalg g$ is not epimorphic.

Now suppose that $f$ and $g$ are epimorphic. Let $h,h' \colon X' \amalg Y' \to Z$ be maps having equal composition with $f \amalg g$. Restricting $h$ and $h'$ to $X'$, we see that they have equal composition with $f$. Since $f$ is epimorphic, this means $h$ and $h'$ have equal restriction to $X'$. Similarly, they have equal restriction to $Y'$. By the definition of co-product, this means $h=h'$, and so $f \amalg g$ is epimorphic.
\end{proof}

\begin{corollary} \label{cor:monic-sum}
For any objects $X$ and $Y$, the natural map $X \to X \amalg Y$ is a monomorphism.
\end{corollary}

\begin{proof}
Let $i$ be the identity map of $X$, and let $j \colon \bzero \to Y$ be the unique map. Clearly, $i$ and $j$ are monomorphisms. The map in question is (isomorphic to) $i \amalg j$, and is thus a monomorphism by Proposition~\ref{prop:monic-sum}(a).
\end{proof}

\begin{proposition} \label{prop:B0-fiber}
Fiber products distribute over co-products, in the following sense. Let $X$, $X'$, and $Y$ be objects of $\cB$ equipped with morphisms to another object $Z$. Suppose that the fiber products $X \times_Z Y$ and $X' \times_Z Y$ exist. Then the fiber product $(X \amalg X') \times_Z Y$ also exists, and the natural map
\begin{displaymath}
(X \times_Z Y) \amalg (X' \times_Z Y) \to (X \amalg X') \times_Z Y
\end{displaymath}
is an isomorphism.
\end{proposition}

\begin{proof}
Let $P=(X \times_Z Y) \amalg (X' \times_Z Y)$ and let $\Phi$ be the functor on $\cB$ given by
\begin{displaymath}
\Phi(W) = \big( \Hom(W, X) \amalg \Hom(W, X') \big) \times_{\Hom(W,Z)} \Hom(W,Y).
\end{displaymath}
Since $P$ has natural maps to $X \amalg X'$ and $Y$ that agree when composed to $Z$, there is a natural transformation $\Hom(-, P) \to \Phi$. It suffices to show that this is an isomorphism, for then $P$ will represent the fiber product. To check that this is an isomorphism, it suffices to verify that $\Hom(W,P) \to \Phi(W)$ is a bijection when $W$ is an atom. In this case, we have natural identifications
\begin{align*}
\Hom(W, P)
=& \Hom(W, X \times_Z Y) \amalg \Hom(W, X' \times_Z Y) \\
=& \big( \Hom(W, X) \times_{\Hom(W,Z)} \Hom(W,Y) \big) \amalg \big( \Hom(W,X') \times_{\Hom(W,Z)} \Hom(W,Y) \big) \\
=& \Phi(W),
\end{align*}
and so the result follows.
\end{proof}

\subsection{Properties of B-categories} \label{ss:B-prop}

We now prove some general results on $\rB$-categories. We fix a $\rB$-category $\cB$ for the duration of \S \ref{ss:B-prop}.

\begin{proposition} \label{prop:subatomic}
The only subobjects of an atom $X$ are $\bzero$ and $X$.
\end{proposition}

\begin{proof}
Suppose that $Y$ is a non-empty subobject of $X$. Write $Y=Y_1 \amalg \cdots \amalg Y_n$ with each $Y_i$ an atom and $n \ge 1$. Since $Y_i \to Y$ is monic by Corollary~\ref{cor:monic-sum}, it follows that $Y_i \to X$ is monic, and thus an isomorphism by Definition~\ref{defn:B}(e). It now follows that $n=1$, since the map $X \amalg X \to X$ is not monic (the two natural maps $X \to X \amalg X$ are distinct by Definition~\ref{defn:B}(c), but have equal composition to $X$). This completes the proof.
\end{proof}

\begin{proposition} \label{prop:epicatom}
Any map of atoms is epimorphic.
\end{proposition}

\begin{proof}
Let $f \colon X \to Y$ be a map of atoms, and let $g,h \colon Y \to Z$ be maps such that $g \circ f=h \circ f$. Since $\cB$ has finite limits, the equalizer $\Eq(g,h)$ of $g$ and $h$ exists, and is naturally a subobject of $Y$. Since $f$ factors through $\Eq(g,h)$ and $X$ is non-empty, it follows that $\Eq(g,h)$ is non-empty (Proposition~\ref{prop:map-to-empty}). Thus $\Eq(g,h)$ is equal to $Y$ (Proposition~\ref{prop:subatomic}), and so $g=h$.
\end{proof}

\begin{proposition} \label{prop:Bmonic}
Let $f \colon X \to Y$ be a morphism. Write $X=X_1 \amalg \cdots \amalg X_n$ and $Y=Y_1 \amalg \cdots \amalg Y_m$ where each $X_i$ and $Y_i$ is atomic. Let $a \colon [n] \to [m]$ and $f_i \colon X_i \to Y_{a(i)}$ be as in Proposition~\ref{prop:Bmap}.
\begin{enumerate}
\item $f$ is epimorphic if and only if $a$ is surjective.
\item $f$ is monomorphic if and only if $a$ is injective and each $f_i$ is an isomorphism.
\end{enumerate}
\end{proposition}

\begin{proof}
For $j \in [m]$, let $X^j = \coprod_{a(i)=j} X_i$, and let $f^j \colon X^j \to Y_j$ be the restriction of $f$. Then $f$ is the co-product of the $f^j$, and so by Proposition~\ref{prop:monic-sum}, $f$ is monomorphic (resp.\ epimorphic) if and only if each $f^j$ is.

(a) Suppose that $f$ is monomorphic. Then each $f^j$ is monomorphic, and so by Proposition~\ref{prop:subatomic} either $X^j$ is empty or $f^j$ is an isomorphism. It follows that $a$ is injective and each $f_i$ is an isomorphism. Conversely, suppose that $a$ is injective and each $f_i$ is an isomorphism. Then each $f^j$ is clearly monomorphic, and so $f$ is too.

(b) Suppose that $f$ is epimorphic. Then each $f^j$ is epimorphic. It follows that $X^j$ is non-empty, as $\bzero \to Y_j$ is not epimorphic (the two maps $Y_j \to Y_j \amalg Y_j$ are distinct by Definition~\ref{defn:B}(c) but have the same restriction to $\bzero$). Thus $a$ is surjective. Conversely, suppose that $a$ is surjective. Then for each $j \in [m]$ there is some $i$ with $a(i)=j$, and then map $f_i$ is epimorphic by Proposition~\ref{prop:epicatom}. It follows that $f^j$ is epimorphic too. Since this holds for each $j$, we find that $f$ is epimorphic.
\end{proof}

\begin{corollary} \label{cor:balanced}
The category $\cB$ is balanced: a morphism that is both monomorphic and epimorphic is an isomorphism.
\end{corollary}

\begin{proof}
Using notation as in the proposition, if $f$ is monomorphic and epimorphic then $a$ is a bijection and each $f_i$ is an isomorphism, and so $f$ is an isomorphism.
\end{proof}

\begin{corollary} \label{cor:Bsub}
Let $X=X_1 \amalg \cdots \amalg X_n$ with each $X_i$ atomic. For a subset $S$ of $[n]$, let $X_S=\coprod_{i \in S} X_i$. Then every subobject of $X$ is one of the $X_S$, and $X_S \subset X_T$ if and only if $S \subset T$.
\end{corollary}

\begin{proof}
This follows immediately from the structure of monomorphisms given in Proposition~\ref{prop:Bmonic}.
\end{proof}

\begin{corollary} \label{cor:im}
Let $f \colon X \to Y$ be a morphism, and use notation as in Proposition~\ref{prop:Bmap}.
\begin{enumerate}
\item $\im(f)$ exists, and is equal to $\coprod_{j \in \im(a)} Y_j$.
\item $f$ is an epimorphism if and only if $\im(f)=Y$.
\item The map $X \to \im(f)$ is an epimorphism, and a monomorphism if and only if $f$ is.
\end{enumerate}
\end{corollary}

\begin{proof}
This follows from the structure of $f$ given in Proposition~\ref{prop:Bmap}, the characterization of monomorphisms and epimorphisms in Proposition~\ref{prop:Bmonic}, and the classification of subobjects in Corollary~\ref{cor:Bsub}.
\end{proof}

\begin{proposition} \label{prop:mono}
Let $f \colon X \to Y$ be a morphism, and let $\Delta \colon X \to X \times_Y X$ be the diagonal map. The following are equivalent:
\begin{enumerate}
\item $f$ is monomorphic.
\item $\Delta$ is an isomorphism.
\item $\Delta$ is epimorphic.
\end{enumerate}
\end{proposition}

\begin{proof}
In any category, (a) and (b) are equivalent, and (b) implies (c). In a balanced category (such as a $\rB$-category), (c) implies (b) since $\Delta$ is always monomorphic.
\end{proof}

\begin{proposition} \label{prop:Hom-finite}
For any objects $X$ and $Y$, the set $\Hom(X,Y)$ is finite.
\end{proposition}

\begin{proof}
Consider a map $f \colon X \to Y$. Let $\Gamma_f \subset X \times Y$ be the image of $\id_X \times f \colon X \to X \times Y$, and let $p \colon \Gamma_f \to X$ and $q \colon \Gamma_f \to Y$ be the projections. Since $\id_X \times f$ is a monomorphism, it follows from Corollary~\ref{cor:im} that the natural map $X \to \Gamma_f$ is both a monomorphism and an epimorphism, and is thus an isomorphism by Corollary~\ref{cor:balanced}; its inverse is clearly $p$. We thus see that $f=q \circ p^{-1}$, and so $f$ can be recovered from $\Gamma_f$. As $X \times Y$ has only finitely many subobjects (by Corollary~\ref{cor:Bsub}), the result follows.
\end{proof}

\begin{corollary} \label{cor:EI-atom}
Any self-map of an atom is an isomorphism.
\end{corollary}

\begin{proof}
Let $f \colon X \to X$ be a map with $X$ an atom. Then $f$ is an epimorphism (Proposition~\ref{prop:epicatom}), and so $f^* \colon \Hom(X,X) \to \Hom(X,X)$ is injective. Since $\Hom(X,X)$ is finite (Proposition~\ref{prop:Hom-finite}), it follows that $f^*$ is bijective, and so there exists $g \in \Hom(X,X)$ such $g \circ f = \id_X$. Thus $f$ is a monomorphism, and hence an isomorphism (Corollary~\ref{cor:balanced}).
\end{proof}

\subsection{Orbits} \label{ss:orb}

Suppose $G$ is an admissible group and $X$ is a finitary $G$-set. One can then form the orbit space $G \backslash X$, which is a finite set. Passing to orbits is often an important idea.

There is an analog of this construction in our more general categories. Let $\cB$ be a $\rB_0$-category. We define the \defn{orbit set} of $X$, denoted $X^{\orb}$, to be the set of atomic subobjects of $X$. This construction is natural: it follows from Proposition~\ref{prop:Bmap} that a map $f \colon X \to Y$ naturally induces a function $f^{\orb} \colon X^{\orb} \to Y^{\orb}$. We therefore have a functor
\begin{displaymath}
\cB \to \FinSet, \qquad X \mapsto X^{\orb},
\end{displaymath}
where $\FinSet$ is the category of finite sets. We now show how one can read off some properties of a morphism from how it behaves on orbits.

\begin{proposition} \label{prop:orb-mono}
Suppose $\cB$ is a $\rB$-category and $f \colon X \to Y$ is a morphism.
\begin{enumerate}
\item $f$ is epimorphic if and only if $f^{\orb}$ is surjective.
\item $f$ is monomorphic if and only if $X^{\orb} \to (X \times_Y X)^{\orb}$ is surjective (or bijective); in this case, $f^{\orb}$ is injective.
\end{enumerate}
\end{proposition}

\begin{proof}
(a) follows from Proposition~\ref{prop:Bmonic}(a). We now prove (b). Let $\Delta \colon X \to X \times_Y X$ be the diagonal. If $f$ is monomorphic then $\Delta$ is an isomorphism (Proposition~\ref{prop:mono}), and so $\Delta^{\orb}$ is a bijection; conversely, if $\Delta^{\orb}$ is surjective then $\Delta$ is epimorphic by (a), and so $f$ is monomorphic (Proposition~\ref{prop:mono}). If $f$ is monomorphic then $f^{\orb}$ is injective by Proposition~\ref{prop:Bmonic}(b).
\end{proof}

\begin{remark}
Let $\cB$ be a $\rB_1$-category. One can sometimes modify $\cB$ to produce a $\rB$-category, as we now describe. Let $f \colon X \to Y$ be a morphism in $\cB$. We make the following definitions:
\begin{itemize}
\item $f$ is a \defn{pre-monomorphism} if the map $X^{\orb} \to (X \times_Y X)^{\orb}$ is bijective.
\item $f$ is a \defn{pre-epimorphism} if the map $X^{\orb} \to Y^{\orb}$ is surjective.
\item $f$ is a \defn{pre-isomorphism} if it is a pre-monomorphism and pre-isomorphism.
\end{itemize}
Suppose that the class of pre-isomorphisms is stable under base change. Then this class forms a right multiplicative system, as defined in \stacks{04VC}. The localized category is a $\rB$-category, and is the universal $\rB$-category to which $\cB$ maps (with respect to functors that preserve finite co-products, finite limits, and atoms).
\end{remark}

\section{Pre-Galois categories} \label{s:pre-galois}

In this section, we identify a few categorical properties of $\bS(G)$ that need not hold for a general $\rB$-category, the most important of which is non-degeneracy. Motivated by these observations, we introduce the class of pre-Galois categories. We also discuss how they relate to the existing notion of Galois category. All categories in this section are assumed to be essentially small.

\subsection{Non-degeneracy}

We begin with the following observation.

\begin{proposition} \label{prop:nondegen}
Let $\cB$ be a $\rB_1$-category. The following are equivalent:
\begin{enumerate}
\item If $X \to Z$ and $Y \to Z$ are maps of atoms then $X \times_Z Y$ is non-empty.
\item A base change of an epimorphism is an epimorphism.
\item A product of epimorphisms is an epimorphism.
\end{enumerate}
\end{proposition}

\begin{proof}
(a) $\Rightarrow$ (b). Let $f \colon X \to Y$ be an epimorphism, let $Y' \to Y$ be an arbitrary map, and let $f' \colon X' \to Y'$ be the base change of $f$. We show that $f'$ is an epimorphism. Since fiber products distribute over co-products, it suffices to treat the case where $X$, $Y$, and $Y'$ are atoms. By assumption, $X'$ is then non-empty, and so $f'$ is an epimorphism.

(b) $\Rightarrow$ (c). Let $X \to Y$ and $X' \to Y'$ be epimorphisms. Consider the composition
\begin{displaymath}
X \times X' \to Y \times X' \to Y \times Y'.
\end{displaymath}
The first map is the base change of the epimorphism $X \to Y$ along the map $X' \to \bone$, and is thus an epimorphism; similarly, the second map is the base change of the epimorphism $X' \to Y'$ along the map $Y \to \bone$, and is thus an epimorphism. It follows that the composition $X \times X' \to Y \times Y'$ is an epimorphism, as required.

(c) $\Rightarrow$ (a). Let $X \to Y$ and $Y' \to Y$ be maps of atoms, and let $X'=X \times_Y Y'$ be the fiber product. Since $X \to Y$ is an epimorphism, by assumption $X' \to Y'$ is also an epimorphism. Thus $X'$ is non-empty.
\end{proof}

Motivated by the above proposition, we make the following definition.

\begin{definition} \label{defn:non-degen}
A $\rB_1$-category is \defn{non-degenerate} if the equivalent conditions of Proposition~\ref{prop:nondegen} hold, and the final object $\bone$ is atomic.
\end{definition}

It is clear that the category $\bS(G; \sE)$ is non-degenerate, for any admissible group $G$ and stabilizer class $\sE$. In Example~\ref{ex:separable}, we give an interesting example of a degenerate $\rB$-category.

\subsection{Implications of non-degeneracy}

Fix a non-degenerate $\rB$-category $\cB$. We now examine some consequences of the non-degeneracy condition. We note that these results can be deduced from the classification of such categories (provided by Theorem~\ref{thm:fraisseB}), but we find it instructive to give direct proofs. For a morphism $f \colon X \to Y$, we define the \defn{kernel pair} of $f$ to be $\Eq(f)=X \times_Y X$. It is a subobject of $X \times X$.

\begin{proposition} \label{prop:non-degen-factor}
Let $f \colon X \to Y$ and $g \colon X \to Z$ be epimorphisms. Then $f$ factors through $g$ if and only if $\Eq(g) \subset \Eq(f)$.
\end{proposition}

\begin{proof}
It is clear that if $f$ factors through $g$ then $\Eq(g) \subset \Eq(f)$. We now prove the converse; thus assume $\Eq(g) \subset \Eq(f)$. Let $I$ be the image of $X$ in $Y \times Z$, and let $h \colon X \to I$, $p \colon I \to Y$, and $q \colon I \to Z$ be the natural maps; note that $f=p \circ h$ and $g=q \circ h$. We have
\begin{displaymath}
\Eq(h)=\Eq(f \times g) = \Eq(f) \cap \Eq(g)=\Eq(g),
\end{displaymath}
where $f \times g$ denotes the map $X \to Y \times Z$. Consider the commutative diagram
\begin{displaymath}
\xymatrix{
X \times_I X \ar[r] \ar[d] & X \times_Z X \ar[d] \\
I \ar[r] & I \times_Z I  }
\end{displaymath}
The top map is the inclusion $\Eq(h) \subset \Eq(g)$, which is an isomorphism. The right map is an epimorphism since $h$ is an epimorphism and the category $\cB$ is non-degenerate; to be a little more precise, note that this morphism is the base change of $X \times X \to I \times I$ along the diagonal $Z \to Z \times Z$. It follows that the bottom map is an epimorphism, and so $q$ is an monomorphism (Proposition~\ref{prop:mono}), and thus an isomorphism (Corollary~\ref{cor:balanced}). We thus have $f=p \circ q^{-1} \circ g$, which completes the proof.
\end{proof}

\begin{corollary} \label{cor:regular-epi}
Let $f \colon X \to Y$ be an epimorphism. Then $f$ is the co-equalizer of the natural maps $X \times_Y X \rightrightarrows X$.
\end{corollary}

\begin{proof}
Let $p_i \colon X \times_Y X \to X$ for $i=1,2$ be the projection maps. Suppose that $g \colon X \to Z$ is a morphism such that $gp_1=gp_1$. We show that $g$ factors through $f$. Replacing $Z$ with the image of $g$, we may assume that $g$ is an epimorphism. The hypothesis on $g$ implies $\Eq(g) \subset \Eq(f)$, and so the result follows from the proposition.
\end{proof}

\begin{corollary} \label{cor:finite-quot}
For $X$ fixed, there are finitely many epimorphisms $X \to Y$ up to isomorphism.
\end{corollary}

\begin{proof}
By the proposition, an epimorphism $f \colon X \to Y$ is determined up to isomorphism by $\Eq(f)$, which is a subobject of $X \times X$. Since $X \times X$ has finitely many subobjects (Corollary~\ref{cor:Bsub}), the result follows.
\end{proof}

\begin{proposition} \label{prop:B-co-compl}
A non-degenerate $\rB$-category $\cB$ is finitely co-complete.
\end{proposition}

\begin{proof}
Since $\cB$ has finite co-products, it suffices to show that it has co-equalizers. Let $f,g \colon X \to Y$ be parallel morphisms. Let $\{ q_i \colon Y \to Z_i \}_{i \in U}$ be representatives of the isomorphism classes of epimorphisms out of $Y$; this set is finite by Corollary~\ref{cor:finite-quot}. Let $V$ be the set of indices $i \in U$ such that $q_i \circ f = q_i \circ g$. Define $I$ to be the image of the map $Y \to \prod_{i \in V} Z_i$, and let $h \colon Y \to I$ be the natural map. We claim that $h$ is a co-equalizer of $(f,g)$.

To see this, suppose that $a \colon Y \to T$ is a morphism with $a \circ f = a \circ g$. The morphism $a$ factors as $c \circ b$, where $b$ is an epimorphism and $c$ is a monomorphism; we may as well assume $b=q_i$ for some $i \in U$. Since $c$ is a monomorphism, it follows that $q_i \circ f = q_i \circ g$, and so $i \in V$. Let $p_i \colon I \to Z_i$ be the projection onto the $i$th factor, so that $q_i=p_i \circ h$. Composing with $c$, we have $a=c \circ p_i \circ h$. We thus see that $a$ factors through $h$. The factorization is unique since $h$ is an epimorphism.
\end{proof}

\begin{remark}
The above proof actually shows that any $\rB$-category satisfying Corollary~\ref{cor:finite-quot} is finitely co-complete. All $\rB$-categories we know (including the degenerate ones) satisfy this corollary.
\end{remark}

\subsection{Effective equivalence relations}

Let $G$ be an admissible group and let $\sE$ be a stabilizer class. By Proposition~\ref{prop:B-co-compl}, the category $\bS(G; \sE)$ is finitely co-complete. This is somewhat surprising, since every smooth $G$-set is a quotient of some $\sE$-smooth $G$-set. The explanation here is that co-equalizers in $\bS(G; \sE)$ do not agree with co-equalizers in $\bS(G)$. In fact, $\bS(G; \sE)$ is a reflective subcategory of $\bS(G)$, and co-equalizers in $\bS(G; \sE)$ are obtained by computing in $\bS(G)$ and then applying the reflector. We now give an example to illustrate the situation.

\begin{example} \label{ex:ineffective}
Let $G=\fS$ be the infinite symmetric group acting on $\Omega=\{1,2,\ldots\}$. Let $\Omega^{[2]}$ be the subset of $\Omega^2$ consisting of pairs $(x,y)$ with $x \ne y$ and let $\Omega^{(2)}$ be the set of 2-element subsets of $\Omega$. Let $p \colon \Omega^{[2]} \to \Omega^{(2)}$ be the natural surjection, and let $R=\Eq(p)$ be the kernel-pair of $p$. In the category $\bS(G)$, the co-equalizer of $R \rightrightarrows \Omega^{[2]}$ is $\Omega^{(2)}$.

Now, let $\sE$ be the stabilizer class consisting of subgroups conjugate to some $\fS(n)$, as in Example~\ref{ex:sym-stab-class}. The $G$-sets $\Omega^{[2]}$ and $R$ are $\sE$-smooth, while $\Omega^{(2)}$ is not. The reflector $\Phi \colon \bS(G) \to \bS(G; \sE)$ is computed on transitive $G$-sets by $\Phi(G/U)=G/V$, where $V$ is the minimal open subgroup over $U$ that belongs to $\sE$ (it is not difficult to see directly that such a subgroup exists). We have $\Omega^{(2)} \cong G/U$, where $U=\fS_2 \times \fS(2)$. From the classification of open subgroups of $\fS$ (see, e.g., \cite[Proposition~15.1]{repst}), we see that the only subgroup in $\sE$ containing $U$ is $\fS$ itself. Thus $\Phi(\Omega^{(2)})=\bone$ is the one-point set, and this is the co-equalizer of $R \rightrightarrows \Omega^{[2]}$ in the category $\bS(\fS; \sE)$.
\end{example}

The following terminology is useful for explaining this situation:

\begin{definition} \label{defn:eff-eq}
Let $\cC$ be a finitely complete category. We say that an equivalence relation $R$ on an object $X$ is \defn{effective} if the quotient $X/R$ exists (this is defined as the co-equalizer of $R \rightrightarrows X$), and the kernel pair of the quotient map $X \to X/R$ is $R$ itself. We say that $\cC$ \defn{has effective equivalence relations} if all equivalence relations in $\cC$ are effective.
\end{definition}

With this terminology, Example~\ref{ex:ineffective} can be summarized as follows: $R$ is an effective equivalence relation in $\bS(\fS)$, but not in the subcategory $\bS(\fS; \sE)$. The following proposition gives the general statement in this direction.

\begin{proposition} \label{prop:ad-eff-eq}
Let $G$ be an admissible group and let $\sE$ be a stabilizer class.
\begin{enumerate}
\item The category $\bS(G)$ has effective equivalence relations.
\item If $\bS(G; \sE)$ has effective equivalence relations then $\bS(G; \sE)=\bS(G)$, i.e., $\sE$ contains all open subgroups of $G$.
\end{enumerate}
\end{proposition}

\begin{proof}
(a) The category of sets has effective equivalence relations. This property passes to $\bS(G)$ since finite limits and co-limits here are computed on the underlying sets.

(b) Let $U$ be an open subgroup of $G$, and let $V$ be a member of $\sE$ with $V \subset U$. Put $Y=G/V$ and $X=G/U$, let $\pi \colon Y \to X$ be the natural map, and let $R \subset Y \times Y$ be the kernel pair of $\pi$. Since $Y \times Y$ belongs to $\bS(G;\sE)$, so does the subobject $R$, and so $R$ defines an equivalence relation on $Y$ in the category $\bS(G;\sE)$. Thus, by assumption, there is a map $\pi' \colon Y \to X'$ in $\bS(G;\sE)$ with kernel pair $R$; of course, we may as well assume $\pi'$ is surjective. Since the inclusion of $\bS(G;\sE)$ into $\bS(G)$ preserves fiber products, it follows that $R$ is the kernel pair of $\pi'$ in $\bS(G)$. Thus $\pi$ and $\pi'$ are isomorphic. In particular, $G/V$ is $\sE$-smooth, and so $V$ belongs to $\sE$.
\end{proof}

\subsection{Pre-Galois categories}

We now introduce this class of categories:

\begin{definition}
A \defn{pre-Galois category} is a non-degenerate $\rB$-category with effective equivalence relations.
\end{definition}

This definition is equivalent to the one given in the introduction. As the preceding discussion shows, if $G$ is an admissible group then $\bS(G)$ is a pre-Galois category.

\subsection{Comparison with Galois categories}

We now discuss the relation between the classical notion of Galois category and our notion of pre-Galois category. We begin by recalling the former:

\begin{definition} \label{def:galois}
A \defn{Galois category} is a pair $(\cC, \omega)$ where $\cC$ is a category and $\omega \colon \cC \to \FinSet$ is a functor (the \defn{fiber functor}) such that the following axioms hold:
\begin{enumerate}
\item $\cC$ has finite limits and colimits.
\item Every morphism $X \to Y$ in $\cC$ factors as $X \to I \to Y$, where $I$ is a summand of $Y$ and $X \to I$ is a strict epimorphim, i.e., $X \to I$ is the co-equalizer of $X \times_I X \rightrightarrows X$.
\item $\omega$ is exact, i.e., it commutes with finite limits and co-limits.
\item $\omega$ is conservative, i.e., $\omega(\phi)$ is an isomorphism if and only if $\phi$ is.
\end{enumerate}
We note there are other axiomizations; this one comes from \cite[\S 2.1.1]{Cadoret}.
\end{definition}

The following is the main result we are after.

\begin{proposition}
Let $\cB$ be a category and $\omega \colon \cB \to \FinSet$ a functor. The following are equivalent:
\begin{enumerate}[(i)]
\item $(\cB, \omega)$ is a Galois category.
\item $\cB$ is a pre-Galois category and $\omega$ is exact and conservative.
\end{enumerate}
\end{proposition}

\begin{proof}
Suppose (i) holds. By the main theorem of Galois categories \cite[Theorem~2.8]{Cadoret}, up to equivalence, $\cB$ is the category of finite $G$-sets, for some pro-finite group $G$, and $\omega$ is the forgetful functor. Since $G$ is an admissible group and $\cB=\bS(G)$, it follows that $\cB$ is pre-Galois. Thus (ii) holds.

Now suppose (ii) holds. We verify the conditions of Definition~\ref{def:galois}. Conditions (c) and (d) hold by assumption. Any $\rB$-category is finitely complete by definition, and a non-degenerate one is finitely co-complete by Proposition~\ref{prop:B-co-compl}; thus (a) holds. Every morphism $f$ in a $\rB$-category factors as $f=g \circ h$, where $h$ is an epimorphism and $g$ is the inclusion of a summand. Thus to complete the proof of (b), it suffices to show that every epimorphism is strict.

Let $f \colon X \to Y$ be an epimorphism, and let $R=\Eq(f)$ be its kernel pair. Since equivalence relations are effective, the quotient $g \colon X \to X/R$ exists, and $R=\Eq(g)$. By Proposition~\ref{prop:non-degen-factor}, we see that $g$ and $f$ are isomorphic. Since $g$ is the co-kernel of $R$, so is $f$, i.e., $f$ is strict.
\end{proof}

The proposition can be summarized as: ``Galois = pre-Galois + fiber functor.''

\section{Categories of atoms} \label{s:Acat}

A $\rB$-category is completely determined by its atoms. In this section, we make this statement precise: we introduce the notion of an $\rA$-category, and show that $\rA$-categories are exactly the (opposite) categories of atoms in a $\rB$-categories. The $\rA$-category perspective is useful since it provides a bridge between $\rB$-categories and finite relational structures. All categories in this section are assumed to be essentially small.

\subsection{The $\AA$ and $\BB$ constructions} \label{ss:AA}

Let $\cB$ be a B${}_0$-category. We define $\AA(\cB)$ to be the full subcategory of $\cB^{\op}$ spanned by the atoms of $\cB$. For example, if $\cB=\bS(G)$ then $\AA(\cB)=\bT(G)^{\op}$ is the opposite of the category of transitive $G$-sets.

Let $\cA$ be an essentially small category. We define a category $\BB(\cA)$ as follows. An object of $\BB(\cA)$ is a finite sequence $X_{\bullet}=(X_1, \ldots, X_n)$ where $X_i$ is an object of $\cA$. A morphism $(X_1, \ldots, X_n) \to (Y_1, \ldots, Y_m)$ consists of a function $a \colon [n] \to [m]$ together with a morphism $X_i \to Y_{a(i)}$ in $\cA^{\op}$ for each $i \in [n]$. Composition is defined in the obvious manner.

\begin{proposition}
For any B${}_0$-category $\cB$, we have an equivalence $\Phi \colon \BB(\AA(\cB)) \to \cB$ given on objects by
\begin{displaymath}
\Phi((X_1, \ldots, X_n)) = X_1 \amalg \cdots \amalg X_n.
\end{displaymath}
\end{proposition}

\begin{proof}
This follows from the basic properties of $\rB_0$-categories established in \S \ref{ss:B0-prop}.
\end{proof}

\begin{proposition} \label{prop:B0}
For any category $\cA$, the category $\BB(\cA)$ is a B${}_0$-category and we have a natural equivalence $\cA \cong \AA(\BB(\cA))$.
\end{proposition}

\begin{proof}
(a) It is clear that co-products in $\BB(\cA)$ are given on objects by
\begin{displaymath}
(X_1, \ldots, X_m) \amalg (Y_1, \ldots, Y_n) = (X_1, \ldots, X_m, Y_1, \ldots, Y_n),
\end{displaymath}
with the obvious structure maps. We note that the zero object of $\BB(\cA)$ is the empty sequence $()$.

(b) Suppose that $X_{\bullet}=(X_1, \ldots, X_n)$ and $Y_{\bullet}=(Y_1, \ldots, Y_m)$ are isomorphic objects of $\BB(\cA)$. Let $(a,f) \colon X_{\bullet} \to Y_{\bullet}$ be the given isomorphism, where $a \colon [n] \to [m]$ is a map of sets and $f_i \colon X_i \to Y_{a(i)}$ is a morphism in $\cA^{\op}$, and let $(b,g) \colon Y_{\bullet} \to X_{\bullet}$ be its inverse. Since the composition is the identity, it follows that $b \circ a$ and $a \circ b$ are the identity maps of $[n]$ and $[m]$; thus $n=m$ and $a$ and $b$ are inverse permutations. Moreover, $f_i \colon X_i \to Y_{a(i)}$ is an isomorphism with inverse $g_{a(i)}$.

From the above, together with the description of the co-product on $\BB(\cA)$, it follows that $(X)$ is an atomic object of $\BB(\cA)$, for any object $X$ of $\cA$. We thus see that every object of $\BB(\cA)$ is a finite co-product of atomic objects.

(c) It follows from the definition of morphisms in $\BB(\cA)$ that the natural map
\begin{displaymath}
\Hom_{\BB(\cA)}((X), Y_{\bullet} \amalg Z_{\bullet}) \to \Hom_{\BB(\cA)}((X), Y_{\bullet}) \amalg \Hom_{\BB(\cA)}((X), Z_{\bullet})
\end{displaymath}
is bijective, for any object $X$ of $\cA$ and objects $Y_{\bullet}$ and $Z_{\bullet}$ of $\BB(\cA)$.
\end{proof}

We thus see that there is a correspondence between $\rB_0$-categories and all (essentially small) categories. In the remainder of this section, we refine this correspondence, and determine what $\rB_1$- and B-categories correspond to. To this end, we begin with one simple observation:

\begin{proposition} \label{prop:AB-epi-mono}
Let $f \colon X \to Y$ be a morphism in the category $\cA$, and let $f' \colon (Y) \to (X)$ be the corresponding morphism in $\BB(\cA)$. Then $f$ is an isomorphism (resp.\ monomorphism, epimorphism) if and only if $f'$ is an isomorphism (resp.\ epimorphism, monomorphism).
\end{proposition}

\begin{proof}
The statement for isomorphisms is clear, as an inverse to one of $f$ or $f'$ gives an inverse to the other. It is also clear that if $f$ is not a monomorphism then $f'$ is not an epimorphism, as a witness to the failure of the former leads to one for the latter. Similarly, it is clear that if $f$ is not an epimorphism then $f'$ is not a monomorphism.

Now suppose that $f'$ is not a monomorphism. Then there exist distinct morphisms $g', h' \colon (Z_1,\ldots,Z_n) \to (Y)$ such that $f' \circ g'=f' \circ h'$. Let $g'_i$ and $h'_i$ be the components of $g'$ and $h'$, and let $g_i$ and $h_i$ be the corresponding morphisms in $\cA$. Since $g' \ne h'$ there is some $i$ such that $g'_i \ne h'_i$. Thus $g_i$ and $h_i$ are distinct morphisms in $\cA$ with $g_i \circ f = h_i \circ f$, and so $f$ is not an epimorphism.

Finally, suppose that $f'$ is not an epimorphism. Then there exist distinct morphisms $g', h' \colon (X) \to (W_1, \ldots, W_n)$ such that $g' \circ f' = h' \circ f'$. By definition, $g'$ corresponds to a morphism $g \colon W_i \to X$ for some $i$, and $h'$ to a morphism $h \colon W_j \to X$ for some $j$. The equality $g' \circ f'=h' \circ f'$ exactly means that $i=j$ and $g \circ f = h \circ f$. Since $g' \ne h'$ we have $g \ne h$, and so $f$ is not a monomorphism.
\end{proof}


\subsection{Initial objects}

Let $\cA$ be a category. We say that a set $S$ of objects of $\cA$ is an \defn{initial set} if for every object $X$ of $\cA$ there exists a unique object $I$ of $S$ such that $\Hom_{\cA}(I,X)$ is non-empty, and this set contains a single element. Suppose $\cA$ has an initial set $S$. For $I \in S$, let $\cA_I$ be the full subcategory of $\cA$ spanned by objects $X$ for which there exists a map $I \to X$. Then $\cA_I$ has $I$ as an initial object, and $\cA$ is the disjoint union of the $\cA_I$'s (as a category). Conversely, if $\cA$ is a (set-indexed) disjoint union of categories with initial objects, then $\cA$ has an initial set.

\begin{proposition} \label{prop:initial}
Let $\cA$ be a category, and let $\cB=\BB(\cA)$.
\begin{enumerate}
\item $\cA$ has a finite initial set if and only if $\cB$ has a final object.
\item $\cA$ has an initial object if and only if $\cB$ has an atomic final object.
\end{enumerate}
\end{proposition}

\begin{proof}
(a) Suppose that $\{I_1, \ldots, I_n\}$ is an initial object of $\cA$. We claim that $I_{\bullet}=(I_1, \ldots, I_n)$ is a final object of $\cB$. Indeed, let $X_{\bullet}=(X_1, \ldots, X_m)$ be given. For each $1 \le i \le m$ there is a unique $1 \le a(i) \le n$ such $\Hom_{\cA}(I_{a(i)}, X_i)$ is non-empty, and it contains a single element $f_i$. The map $a$ together with $f_1, \ldots, f_n$ define a morphism $X_{\bullet} \to I_{\bullet}$ in $\cB$, and it is clearly the unique such map. Thus $I_{\bullet}$ is a final object of $\cB$. This reasoning is reversible too: if $I_{\bullet}$ is a final object of $\cB$ then $\{I_1, \ldots, I_n\}$ is an initial set of $\cA$.

(b) This is clear from the proof of (a).
\end{proof}

\subsection{Amalgamations} \label{ss:amalg}

A \defn{pre-amalgamation} in $\cA$ is a pair of morphisms $(b \colon A \to B, c \colon A \to C)$. Given a pre-amalgamation $(b,c)$, define $\Amalg(b,c)$ to be the category whose objects are pairs $(b' \colon B \to D, c' \colon C \to D)$ of morphisms in $\cA$ with $b'b=c'c$, with the obvious morphisms. An \defn{amalgamation set} for $(b,c)$ is an initial set of this category; we call the elements of this set \defn{amalgamations}.

\begin{proposition} \label{prop:amalg}
Let $\cA$ be a category and let $\cB=\BB(\cA)$ be the corresponding $\rB_0$-category. The following are equivalent:
\begin{enumerate}
\item Every pre-amalgamation in $\cA$ has a finite amalgamation set.
\item The category $\cB$ has fiber products.
\end{enumerate}
\end{proposition}

\begin{proof}
Suppose (b) holds. Let $(b,c)$ be a pre-amalgamation in $\cA$, where $b \colon A \to B$ and $c \colon A \to C$. Let $(X_1, \ldots, X_n)$ be the fiber product of $(B)$ with $(C)$ over $(A)$ in $\cB$. The map $(X_1, \ldots, X_n) \to (B)$ in $\cB$ corresponds to morphisms $f_i \colon B \to X_i$ in $\cA$, for $1 \le i \le n$. Similarly, the map $(X_1, \ldots, X_n) \to C$ corresponds to morphisms $g_i \colon X_i \to C$ in $\cA$, for $1 \le i \le n$. Clearly, $f_i \circ a = g_i \circ b$, so each $(f_i, g_i)$ is an object of $\Amalg(b,c)$.

We claim that $S=\{(f_i,g_i)\}_{1 \le i \le n}$ is an amalgamation set for $(b,c)$. Thus let $(f \colon B \to Y, g \colon C \to Y)$ be an arbitrary object of $\Amalg(b,c)$. Then $f$ defines a morphism $(Y) \to (B)$ in $\cB$, and similarly, $g$ defines a morphism $(Y) \to (C)$ in $\cB$. The two composition to $(A)$ agree, and so there is a unique morphism $(Y) \to (X_1, \ldots, X_n)$ that composes with the projections to the given morphisms. This proves the claim, and so (a) holds.

Now suppose (a) holds. Let $(B) \to (A)$ and $(C) \to (A)$ be morphisms of atoms in $\cB$, corresponding to maps $b \colon A \to B$ and $c \colon A \to C$ in $\cA$. Let $\{(f_i, g_i)\}_{1 \le i \le n}$ be an amalgamation set for $(b,c)$, where $f_i$ and $g_i$ map to $X_i$. Then, reversing the above reasoning, we see that $(X_1, \ldots, X_n)$ is naturally the fiber product of $(B)$ and $(C)$ over $(A)$.

We thus find that the fiber product of morphisms of atoms in $\cB$ always exists. It follows from Proposition~\ref{prop:B0-fiber} that all fibers products exist.
\end{proof}

We say that a category $\cA$ has the \defn{amalgamation property} (AP) if for every pre-amalgamation $(b,c)$ the category $\Amalg(b,c)$ is non-empty. This means that every diagram
\begin{displaymath}
\xymatrix{
B \ar@{..>}[r] & D \\
A \ar[u]^b \ar[r]^c & C \ar@{..>}[u] }
\end{displaymath}
can be filled, i.e., one can find $D$ and the dotted arrows making the square commute.

\begin{proposition} \label{prop:amalg2}
Let $\cA$ be a category in which all pre-amalgamations have a finite amalgamation set, and let $\cB=\BB(\cA)$. Then $\cA$ has the amalgamation property if and only if for every morphism of atoms $X \to Z$ and $Y \to Z$ in $\cB$, the fiber product $X \times_Z Y$ is non-empty.
\end{proposition}

\begin{proof}
This is clear from the proof of Proposition~\ref{prop:amalg}.
\end{proof}

\subsection{A-categories} \label{ss:Acat}

We are finally ready to introduce the main concept of this section:

\begin{definition} \label{defn:Acat}
An \defn{A-category} is an essentially small category $\cA$ satisfying the following conditions:
\begin{enumerate}
\item The category $\cA$ has a finite initial set.
\item Every pre-amalgamation has a finite amalgamation set.
\item Every epimorphism in $\cA$ is an isomorphism.
\end{enumerate}
An \defn{$\rA_1$-category} is an essentially small category satisfying conditions (a) and (b).
\end{definition}

The following is the main result of this section:

\begin{theorem} \label{thm:AB}
Let $\cA$ be a category and put $\cB=\BB(\cA)$.
\begin{enumerate}
\item $\cB$ is a $\rB_1$-category $\iff$ $\cA$ is an $\rA_1$-category.
\item $\cB$ is a $\rB$-category $\iff$ $\cA$ is an $\rA$-category.
\item $\cB$ is a non-degenerate $\rB$-category $\iff$ $\cA$ is an $\rA$-category with an initial object and the amalgamation property.
\end{enumerate}
\end{theorem}

\begin{proof}
(a) follows from Propositions~\ref{prop:B0}, \ref{prop:initial}(a), and~\ref{prop:amalg}; (b) then follows from Proposition~\ref{prop:AB-epi-mono}; and (c) then follows from Propositions~\ref{prop:initial}(b) and~\ref{prop:amalg2}.
\end{proof}

\begin{corollary} \label{cor:A-mono}
All morphisms in an $\rA$-category are monomorphisms.
\end{corollary}

\begin{proof}
This follows from Propositions~\ref{prop:epicatom} and~\ref{prop:AB-epi-mono}.
\end{proof}

\begin{corollary} \label{cor:A-EI}
Any endomorphism in an $\rA$-category is an isomorphism.
\end{corollary}

\begin{proof}
This follows from Corollary~\ref{cor:EI-atom}  and~\ref{prop:AB-epi-mono}.
\end{proof}

\begin{remark}
A category in which all endomorphisms are isomorphisms is called an \defn{EI-category}. Thus the above corollary shows that every $\rA$-category is an EI-category. Representations of EI-categories have received some attention in the literature, e.g., \cite{GanLi}.
\end{remark}

We now discuss the condition Definition~\ref{defn:Acat}(c) in a bit more detail. The contrapositive of Definition~\ref{defn:Acat}(c) can be phrased as follows: if $f \colon X \to Y$ is a non-isomorphism then there exist distinct morphisms $g_1,g_2 \colon Y \to Z$ such that $g_1 \circ f = g_2 \circ f$. As Corollary~\ref{cor:A-mono} suggests, when working on the ``$\rA$ side,'' morphisms will in some sense be embeddings. From this perspective, Definition~\ref{defn:Acat}(c) essentially means that if $X$ is a proper subobject of $Y$ then we can find distinct embeddings of $Y$ into some auxiliary object that agree on $X$.

There is one other perspective on Definition~\ref{defn:Acat}(c) that is sometimes useful. Let $f \colon X \to Y$ be a morphism in an $\rA_1$-category. We refer to objects in the amalgamation set of $(f,f)$ as \defn{self-amalgamations} of $Y$ over $X$. There is always a trivial self-amalgamation, namely $Y$ itself, or more precisely, the pair $(\id_Y, \id_Y)$. One easily sees that $f$ is an epimorphism if and only if this is the only self-amalgamation. Thus the contrapositive of Definition~\ref{defn:Acat}(c) is equivalent to the following: if $f \colon X \to Y$ is a non-isomorphism then there is a non-trivial self-amalgamation of $Y$ over $X$.

\subsection{Products}

Let $\cA_1$ and $\cA_2$ be $\rA_1$-categories. One easily sees that the product category $\cA_1 \times \cA_2$ is also an $\rA_1$-category, and is an $\rA$-category if both $\cA_1$ and $\cA_2$ are. This motivates the following construction:

\begin{definition}
Let $\cB_1$ and $\cB_2$ be $\rB_1$-categories. We define the \defn{tensor product category} to be the $\rB_1$-category
\begin{displaymath}
\cB_1 \boxtimes \cB_2 = \BB(\AA(\cB_1) \times \AA(\cB_2)).
\end{displaymath}
If $\cB_1$ and $\cB_2$ are both $\rB$-categories then so is $\cB_1 \boxtimes \cB_2$.
\end{definition}

\begin{example}
Let $G_1$ and $G_2$ be admissible groups with stabilizer classes $\sE_1$ and $\sE_2$. One can then show
\begin{displaymath}
\bS(G_1; \sE_1) \boxtimes \bS(G_2; \sE_2) \cong \bS(G_1 \times G_2; \sE_1 \times \sE_2),
\end{displaymath}
where $\sE_1 \times \sE_2$ denotes the set of open subgroups of the product of the form $U_1 \times U_2$ with $U_i \in \sE_i$. Note that if $\sE_1$ and $\sE_2$ each contain all open subgroups then the same need not be true for $\sE_1 \times \sE_2$. Thus one is essentially forced to confront stabilizer classes when considering the tensor product construction.
\end{example}

\section{Fra\"iss\'e theory} \label{s:fraisse}

In this section, we review classical Fra\"iss\'e theory and its categorical reformulation, and then apply this theory to prove the main theorems of this paper.

\subsection{Classical Fra\"iss\'e theory} \label{ss:fraisse-class}

We now recall the classical formulation Fra\"iss\'e's theorem. While we will not apply this version of the theorem, it serves as motivation for the categorical form discussed in \S \ref{ss:cat-fraisse} that we do use. We will also use the language of relational structures in \S \ref{s:examples} to construct examples of $\rA$-categories. We refer to \cite{CameronBook} and \cite{Macpherson} for more complete discussions.

A \defn{signature} is a collection $\Sigma=\{(R_i,n_i)\}_{i \in I}$ where $R_i$ is a formal symbol and $n_i$ is a positive integer, called the \defn{arity} of $R_i$. Fix a signature $\Sigma$. A \defn{(relational) structure} for $\Sigma$ is a set $X$ equipped with for each $i \in I$ an $n_i$-ary relation $R_i$ on $X$ (i.e., a subset of $X^{n_i}$). Given a structure $X$ and a subset $Y$, there is an induced structure on $Y$; we call structures obtained in this manner \defn{substructures} of $X$. An \defn{embedding} of structures $X \to Y$ is an injective function that identifies $X$ with a substructure of $Y$.

A structure $\Omega$ is called \defn{homogeneous} if whenever $X$ and $Y$ are finite substructures and $i \colon X \to Y$ is an isomorphism of structures, there exists an automorphism $\sigma$ of $\Omega$ such that $\sigma(x)=i(x)$ for all $x \in X$. The \defn{age} of a structure $\Omega$, denoted $\age(\Omega)$, is the set of all finite structures that embed into $\Omega$. If $\Omega$ is a countable homogeneous structure then $\cC=\age(\Omega)$ has the following properties:
\begin{itemize}
\item $\cC$ is \defn{hereditary}: if $Y$ belongs to $\cC$ and $X$ is (isomorphic to) a substructure of $Y$ then $X$ belongs to $\cC$.
\item The set $\vert \cC \vert$ of isomorphism classes in $\cC$ is countable.
\item $\cC$ satisfies the amalgamation property, as defined in \S \ref{ss:amalg}; here we treat $\cC$ as a category with morphisms being embeddings.
\end{itemize}
Fra\"iss\'e's theorem is the converse statement: if $\cC$ is a class of finite structures satisfying the above three conditions then $\cC$ is the age of a countable homogeneous structure $\Omega$, which is unique up to isomorphism. A class satisfying the above conditions is called a \defn{Fra\"iss\'e class}, and the resulting homogeneous structure $\Omega$ is called the \defn{Fra\"iss\'e limit} of $\cC$.

For a class $\cC$ of structures, let $\cC_n$ denote the subclass consisting of structures with $n$ elements. Suppose $\Omega$ is a homogeneous structure and $\cC=\age(\Omega)$ has the property that $\vert \cC_n \vert$ is finite for all $n$. Then one easily sees that $G=\Aut(\Omega)$ acts oligomorphically on $\Omega$. In this way, Fra\"iss\'e limits provide a powerful mechanism for constructing oligomorphic groups.

\begin{example}
We give a few examples of Fra\"iss\'e limits.
\begin{enumerate}
\item Take the signature to be empty, so that a structure is simply a set. The class $\cC$ of all finite sets is a Fra\"iss\'e class, and the Fra\"iss\'e limit $\Omega$ is a countable infinite set. The oligomorphic group $G=\Aut(\Omega)$ is the infinite symmetric group.
\item Take the signature to consist of a single binary relation. The class $\cC$ of all finite totally ordered sets is a Fra\"iss\'e class, and the Fra\"iss\'e limit $\Omega$ is the set of rational numbers equipped with its standard total order.
\item Again, take the signature to consist of a single binary relation. Let $\cC$ be the class of all finite simple graphs. This is a Fra\"iss\'e class, and the limit is the Rado graph. \qedhere
\end{enumerate}
\end{example}

\subsection{Categorical Fra\"iss\'e theory} \label{ss:cat-fraisse}

Given a class $\cC$ of relational structures, one can regard $\cC$ as a category with morphisms being embeddings. Fra\"iss\'e's theorem is thus a statement about a certain class of categories. It turns out that the theorem actually holds for a much broader class of categories. This observation goes back to the work of Droste--G\"obel \cite{DrosteGobel1,DrosteGobel2}, and has been discussed in more recent work as well \cite{Caramello,Irwin,Kubis}. We follow the treatment in the appendix to our recent paper \cite{homoten}.

Fix a category $\cC$ in which all objects are monomorphisms; we often refer to morphisms in $\cC$ as embeddings. An \defn{ind-object} in $\cC$ is a diagram $X_1 \to X_2 \to \cdots$ in $\cC$. It is possible to consider ind-objects indexed by more general posets, but we will only need this simple version. There is a natural notion of morphism between ind-objects, and between an ordinary object and an ind-object; see \cite[\S A.2]{homoten}.

Let $\Omega$ be an ind-object of $\cC$. We say that $\Omega$ is \defn{universal} if every object of $\cC$ embeds into $\Omega$. We say that $\Omega$ is \defn{homogeneous} if every isomorphism of finite subobjects is induced by an automorphism. Precisely, this means the following. Suppose $\alpha \colon X \to \Omega$ and $\beta \colon Y \to \Omega$ are embeddings, where $X$ and $Y$ are objects of $\cC$, and that we have an isomorphism $\gamma \colon X \to Y$ in $\cC$. Then there must exist an automorphism $\sigma$ of $\Omega$ such that $\sigma \circ \alpha = \beta \circ \gamma$. We say that $\cC$ is a \defn{Fra\"iss\'e category} if it admits a universal homogeneous ind-object. We note that any two universal homogeneous ind-objects are isomorphic \cite[Proposition~A.7]{homoten}.

Fra\"iss\'e's theorem gives a characterization of Fra\"iss\'e categories. To state it, we will need the amalgamation property (AP) defined in \S \ref{ss:amalg}, as well as the following condition:
\begin{itemize}[leftmargin=1.5cm]
\item[(RCC)] \defn{Relative countable cofinality}: for any object $X$ of $\cC$ there exists a cofinal sequence of morphisms out of $X$, i.e., there is a sequence of morphisms $\{\alpha_n \colon X \to Y_n \}_{n \ge 1}$ such that if $\beta \colon X \to Y$ is any morphism then there is a morphism $\gamma \colon Y \to Y_n$ for some $n$ such that $\gamma \circ \beta = \alpha_n$.
\end{itemize}
The following is the categorical Fra\"iss\'e theorem (in one form).

\begin{theorem}[{\cite[Theorem~A.11]{homoten}}] \label{thm:fraisse0}
Suppose that $\cC$ has an initial object. Then $\cC$ is a Fra\"iss\'e category if and only if (RCC) and (AP) hold.
\end{theorem}

\begin{example}
Here is an example where the categorical Fra\"iss\'e theorem applies while the classical one does not apply. A \defn{cubic space} is a complex vector space $V$ equipped with a linear map $\Sym^3(V) \to \bC$. There is a natural notion of embedding for cubic spaces. In \cite{homoten}, we show that the category of finite dimensional cubic spaces is a Fra\"iss\'e category; we give many other related examples as well.
\end{example}

\subsection{Fra\"iss\'e theory for A-categories}

The following is our main Fra\"iss\'e-like theorem for $\rA$-categories.

\begin{theorem} \label{thm:fraisseA-1}
Let $\cA$ be an A-category satisfying the following conditions:
\begin{itemize}
\item $\cA$ has an initial object.
\item $\cA$ satisfies the amalgamation property.
\item $\cA$ has countably many isomorphism classes.
\end{itemize}
Then there exists an admissible group $G$ and a stabilizer class $\sE$ for $G$ such that $\cA$ is equivalent to $\bT(G;\sE)^{\op}$.
\end{theorem}

We will actually prove a slightly more precise statement. Let $\cA$ be any category satisfying the three conditions of Theorem~\ref{thm:fraisseA-1}. By Theorem~\ref{thm:fraisse0}, the category $\cA$ is Fra\"iss\'e, and thus admits a universal homogeneous ind-object $\Omega$. Let $G$ be its automorphism group. For an object $X$, we let $\Phi(X)$ be the set of all embeddings $X \to \Omega$; note that this is non-empty since $\Omega$ is universal. The group $G$ naturally acts on $\Phi(X)$, via its action on $\Omega$, and this action is transitive by homogeneity. Give $\alpha \in \Phi(X)$, we let $G_{\alpha}$ be the stabilizer of $\alpha$ in $G$. Let $\sE$ be the set of all subgroups of $G$ of the form $G_{\alpha}$, for some $\alpha$.

\begin{theorem} \label{thm:fraisseA-2}
Let $\cA$ be an $\rA_1$-category satisfying the three conditions of Theorem~\ref{thm:fraisseA-1}, and let $\Omega$, $G$, $\sE$, and $\Phi$ be as above.
\begin{enumerate}
\item The family $\sE$ is a neighborhood basis for a first-countable admissible topology on $G$.
\item The family $\sE$ is a stabilizer class for $G$.
\item The construction $\Phi$ defines a faithful and essentially surjective functor $\cA \to \bT(G; \sE)^{\op}$.
\item If $\cA$ is an $\rA$-category then the functor in (c) is an equivalence.
\end{enumerate}
\end{theorem}

\begin{remark}
In \S \ref{ss:matching}, we give an example of an $\rA_1$-category (that is not an $\rA$-category) where the functor in (c) is not an equivalence.
\end{remark}

\begin{remark} \label{rmk:complete}
There is a notion of completeness for admissible groups. In Theorem~\ref{thm:fraisseA-1}, there is in fact a unique (up to isomorphism) complete group satisfying the concluding statement. The group $G$ constructed following the statement of the theorem is this complete group.
\end{remark}

We now prove the theorem, in a series of lemmas. We fix $\cA$, $\Omega$, $G$, $\sE$, and $\Phi$ as in the theorem statement in what follows. We also write $\bone$ for the initial object of $\cA$.

\begin{lemma} \label{lem:fraisse1-1}
Let $X$ and $Y$ be objects of $\cA$, and let $\alpha \colon X \to \Omega$ and $\beta \colon Y \to \Omega$ be embeddings. Then there is a unique (up to isomorphism) diagram
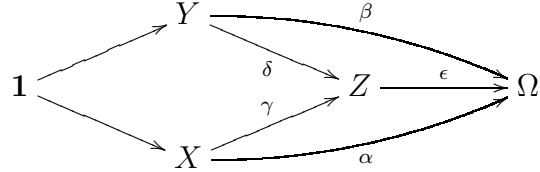
\begin{displaymath}
\xymatrix@C=4em@R=1em{
& Y \ar[rd]_{\delta} \ar@/^8pt/[rrd]^{\beta} \\
\bone \ar[ru] \ar[rd] && Z \ar[r]^{\epsilon} & \Omega \\
& X \ar[ru]^{\gamma} \ar@/_8pt/[rru]_{\alpha} }
\end{displaymath}
where $(Z,\gamma,\delta)$ is an amalgamation of $X$ and $Y$ over the trivial object $\bone$. We have $G_{\epsilon}=G_{\alpha} \cap G_{\beta}$.
\end{lemma}

\begin{proof}
The existence and uniqueness of the diagram follow from the definition of $\rA_1$-category. We have $\alpha=\epsilon \gamma$, and so for $\sigma \in G$ we have $\sigma \alpha = \sigma \epsilon \gamma$; thus $G_{\epsilon} \subset G_{\alpha}$. Of course, the same holds with $\beta$, and so $G_{\epsilon} \subset G_{\alpha} \cap G_{\beta}$. We now prove the reverse containment. Thus let $\sigma \in G_{\alpha} \cap G_{\beta}$ be given. Then the above diagram commutes with $\epsilon$ changed to $\sigma \epsilon$. By uniqueness of the above diagram, it follows that $\epsilon=\sigma \epsilon$, and so $\sigma \in G_{\epsilon}$, as required.
\end{proof}

\begin{lemma} \label{lem:fraisse-6}
Let $X$ and $Y$ be objects of $\cA$, let $E=G \backslash (\Phi(X) \times \Phi(Y))$, and let $F$ be an amalgamation set for $X$ and $Y$ over $\bone$. Then we have a natural bijection $E \cong F$; in particular, $E$ is finite.
\end{lemma}

\begin{proof}
Given $\alpha \in \Phi(X)$ and $\beta \in \Phi(Y)$, let $(Z,\gamma,\delta)$ be the amalgamation from Lemma~\ref{lem:fraisse1-1}. It is clear that if $(\alpha,\beta)$ is modified by an element of $G$ then the amalgamation is unchanged (up to isomorphism). This construction therefore yields a well-defined map $E \to F$. Conversely, if $(Z,\gamma,\delta)$ is any amalgamation then by choosing an embedding $\epsilon \colon Z \to \Omega$, we get the pair $(\gamma^*(\epsilon), \delta^*(\epsilon))$ in $\Phi(X) \times \Phi(Y)$, and the orbit of this pair is independent of the choice fo $\epsilon$. This provides a map $F \to E$. One readily verifies the two maps are inverse to one another. Since $F$ is finite by the definition of $\rA_1$-category, it follows that $E$ is finite.
\end{proof}

\begin{lemma} \label{lem:fraisse1-2}
The set $\sE$ is a neighborhood basis for an admissible topology on $G$, and $\sE$ is a stabilizer class for the admissible group $G$.
\end{lemma}

\begin{proof}
If $\alpha$ is the unique embedding of the trivial object into $\Omega$ then $G_{\alpha}=G$; thus $G$ belongs to $\sE$. It is clear that $\sE$ is closed under conjugation. Lemma~\ref{lem:fraisse1-1} shows that $\sE$ is closed under finite intersections. It follows that $\sE$ is a neighborhood basis for a topology on $G$, and the $\sE$ is a stabilizer class.

It remains to show that the topological group $G$ is admissible. It is non-archimedean by construction. We now verify that it is Hausdorff. Thus suppose $\sigma$ belongs to $\bigcap_{U \in \sE} U$. Then for any embedding $\alpha \colon X \to \Omega$ we have $\sigma \alpha=\alpha$. Since a map of ind-objects is determined by its restrictions to (non-ind) objects, it follows that $\sigma$ is the identity, and so $G$ is Hausdorff.

Finally, we show $G$ is Roelcke pre-compact. It suffices to show $G_{\alpha_0} \backslash G / G_{\beta_0}$ is finite for two embeddings $\alpha_0 \colon X \to \Omega$ and $\beta_0 \colon Y \to \Omega$. This set is in bijection with $G \backslash (G/G_{\alpha_0} \times G/G_{\beta_0})$. Since $G$ acts transitively on $\Phi(X)$ with stabilizer $G_{\alpha_0}$, the set $G/G_{\alpha}$ (with its $G$-action) is identified with $\Phi(X)$; similarly, $G/G_{\beta_0}$ is identified with $\Phi(Y)$. Thus finiteness follows from Lemma~\ref{lem:fraisse-6}.
\end{proof}

We have thus proved Theorem~\ref{thm:fraisseA-2}(a,b). Now, the action of $G$ on $\Phi(X)$ is smooth, by definition of the topology on $G$. If $\alpha \colon X \to Y$ is a morphism in $\cA$ then there is an induced morphism $\alpha^* \colon \Phi(Y) \to \Phi(X)$ of $G$-sets. It follows that we have a functor
\begin{displaymath}
\Phi \colon \cA \to \bT(G)^{\op}.
\end{displaymath}
To complete the proof of the theorem, we study properties of this functor in the next sequence of lemmas.

\begin{lemma}
The functor $\Phi$ is faithful.
\end{lemma}

\begin{proof}
Let $\alpha$ and $\beta$ be two morphisms $X \to Y$ in $\cC$ such that $\alpha^*=\beta^*$. Choose an embedding $\gamma \colon Y \to \Omega$, which is possible since $\Omega$ is universal. By assumption, we have $\gamma \circ \alpha=\gamma \circ \beta$. Since $\gamma$ is a monomorphism, it follows that $\alpha=\beta$. Thus $\Phi$ is faithful.
\end{proof}

\begin{lemma}
The essential image of $\Phi$ is $\bT(G; \sE)$.
\end{lemma}

\begin{proof}
For an object $X$ of $\cA$, the $G$-set $\Phi(X)$ is isomorphic to $G/G_{\alpha}$, where $\alpha \in \Phi(X)$ is any element. We thus see that the essential image of $\Phi$ exactly consists of $G$-sets isomorphic to $G/U$ with $U \in \sE$, which is exactly $\bT(G; \sE)$.
\end{proof}

We have thus proved Theorem~\ref{thm:fraisseA-2}(c). We now turn our attention to Theorem~\ref{thm:fraisseA-2}(d). In what follows, we assume that $\cA$ is an $\rA$-category.

\begin{lemma} \label{lem:fraisse-5}
The functor $\Phi$ is conservative; that is, if $\alpha \colon X \to Y$ is a morphism in $\cC$ such that $\alpha^* \colon \Phi(Y) \to \Phi(X)$ is an isomorphism then $\alpha$ is an isomorphism.
\end{lemma}

\begin{proof}
Since $\cA$ is an A-category, it is enough to show that $\alpha$ is an epimorphism. Thus suppose that $\beta$ and $\gamma$ are maps $Y \to Z$ such that $\beta \circ \alpha = \gamma \circ \alpha$. We thus have $\alpha^*\beta^*=\alpha^*\gamma^*$. Since $\alpha^*$ is an isomorphism, it follows that $\beta^*=\gamma^*$. Since $\Phi$ is faithful, we find $\beta=\gamma$, as required.
\end{proof}

\begin{lemma} \label{lem:fraisse-4}
The functor $\Phi$ is full.
\end{lemma}

\begin{proof}
Let $X$ and $Y$ be objects of $\cC$, and let $\phi \colon \Phi(Y) \to \Phi(X)$ be a map of $G$-sets. Choose an element $\beta \in \Phi(Y)$, and let $\alpha=\phi(\beta)$. Note that since $\phi$ is $G$-equivariant, we have $G_{\beta} \subset G_{\alpha}$. Let $(Z,\gamma,\delta)$ be an amalgamation of $X$ and $Y$ over $\bone$, and let $\epsilon \colon Z \to \Omega$ be an embedding, as in Lemma~\ref{lem:fraisse1-1}. We have $G_{\epsilon}=G_{\alpha} \cap G_{\beta}=G_{\beta}$. Thus $\gamma^* \colon \Phi(Z) \to \Phi(Y)$ is an isomorphism of $G$-sets; indeed, it is a $G$-equivariant map of transitive $G$-sets mapping $\epsilon$ to $\beta$, and $\epsilon$ and $\beta$ have the same stabilizer in $G$. By the Lemma~\ref{lem:fraisse-5}, it follows that $\gamma$ is an isomorphism.

Since the diagram in Lemma~\ref{lem:fraisse1-1} is only defined up to isomorphism, we may as well suppose that $Z=Y$, $\gamma=\id_Y$, and $\beta=\epsilon$. We thus see that $\delta^* \colon \Phi(Y) \to \Phi(X)$ is a map of $G$-sets carrying $\beta$ to $\alpha$. Since $\Phi(Y)$ is transitive, it follows that $\phi=\delta^*$, which completes the proof.
\end{proof}

\subsection{Fra\"iss\'e theory for B-categories}

The following is our main theorem on $\rB$-categories in the countable case.

\begin{theorem} \label{thm:fraisseB}
Let $\cB$ be a $\rB$-category that is non-degenerate and has countably many isomorphism classes. Then there is a first-countable admissible group $G$ and a stabilizer class $\sE$ such that $\cB$ is equivalent to $\bS(G;\sE)$. Moreover, if equivalence relations in $\cB$ are effective (i.e., $\cB$ is pre-Galois) then $\cB$ is equivalent to $\bS(G)$.
\end{theorem}

\begin{proof}
Let $\cA=\AA(\cB)$. By Theorem~\ref{thm:AB}, this is an A-category satisfying the three conditions of Theorem~\ref{thm:fraisseA-1}. Thus by that theorem, we have $\cA \cong \bT(G; \sE)$ for some first-countable admissible group $G$ and stabilizer class $\sE$. We have equivalences $\cB=\BB(\cA)$ and $\bS(G;\sE)=\BB(\bT(G;\sE)^{\op})$, and so we obtain an equivalence $\cB \cong \bS(G;\sE)$. The second statement follows from Proposition~\ref{prop:ad-eff-eq}.
\end{proof}

\subsection{The uncountable case}

We now explain how to remove the countability hypothesis in Theorem~\ref{thm:fraisseB}. Although this works in the non-degenerate case, we simply handle the pre-Galois case for simplicity. The following is the main result we are after:

\begin{theorem} \label{thm:uncountable}
Let $\cB$ be any pre-Galois category. Then there is an admissible group $G$ such that $\cB$ is equivalent to $\bS(G)$.
\end{theorem}

We require a few lemmas before proving the theorem. Let $\cB$ and $\cB'$ be pre-Galois categories. An \defn{embedding} $\Phi \colon \cB \to \cB'$ is a functor that commutes with finite limits and co-products, and maps atoms to atoms.

\begin{lemma} \label{lem:uncountable-1}
Let $\Phi \colon \cB \to \cB'$ be an embedding of pre-Galois categories.
\begin{enumerate}
\item $\Phi$ is fully faithful.
\item The essential image of $\Phi$ is closed under taking quotients, i.e., if $Y'$ is an object of $\cB'$ for which there is an epimorphism $\Phi(X) \to Y'$ for some $X \in \cB$ then $Y'$ is in the essential image of $\Phi$.
\item Suppose that for any $Y' \in \cB'$ there is $X \in \cB$ and an epimorphism $\Phi(X) \to Y'$. Then $\Phi$ is an equivalence.
\end{enumerate}
\end{lemma}

\begin{proof}
(a) We first claim that $\Phi$ is conservative, i.e., if $f$ is a morphism in $\cB$ such that $\Phi(f)$ is an isomorphism then $f$ is an isomorphism. Since $\Phi$ maps atoms to atoms and commutes with finite co-products, we have a natural isomorphism $X^{\orb} \cong \Phi(X)^{\orb}$, where $(-)^{\orb}$ is the orbits functor (see \S \ref{ss:orb}). Now, let $f \colon X \to Y$ be a morphism in $\cB$ such that $\Phi(f)$ is an isomorphism. Let $\Delta_f \colon X \to X \times_Y X$ be the diagonal map. Since $\Phi$ commutes with finite limits, we have $\Phi(\Delta_f)=\Delta_{\Phi(f)}$. Since $\Phi(f)$ is an isomorphism, $\Phi(f)^{\orb}$ is surjective and $\Delta_{\Phi(f)}^{\orb}$ is bijective by Proposition~\ref{prop:orb-mono}. Since $f^{\orb} = \Phi(f)^{\orb}$ and $\Delta_f^{\orb} = \Phi(\Delta_f)^{\orb}=\Delta_{\Phi(f)}^{\orb}$, we see that $f^{\orb}$ is surjective and $\Delta_f^{\orb}$ is bijective. Thus $f$ is a monomorphism and epimorphism by Proposition~\ref{prop:orb-mono}, and therefore an isomorphism by Corollary~\ref{cor:balanced}. This proves the claim.

Now, for $X \in \cB$, subobjects of $X$ correspond to subsets of $X^{\orb}$ (Proposition~\ref{cor:Bsub}). Since $\Phi$ preserves orbits, it follows that $\Phi$ induces a bijection between subobjects of $X$ and subobjects of $\Phi(X)$. Since a morphism $X \to Y$ is determined by its graph, which is a subobject of $X \times Y$, it follows that $\Phi$ is faithful.

We now show that $\Phi$ is full. Thus let $X$ and $Y$ be objects of $\cB$, and let $f' \colon \Phi(X) \to \Phi(Y)$ be a morphism in $\cB'$. Let $\Gamma' \subset \Phi(X \times Y)$ be the graph of $f'$. By what we said above, $\Gamma'$ has the form $\Phi(\Gamma)$ for a unique subobject $\Gamma$ of $X \times Y$. Let $p_1 \colon \Gamma \to X$ be the projection map. Since $\Gamma'$ is a graph, we see that $\Phi(p_1)$ is an isomorphism. Thus $p_1$ is an isomorphism since $\Psi$ is conservative. It follows that $\Gamma$ is the graph of a morphism $f$, and clearly $\Phi(f)=f'$. Thus $\Phi$ is full.

(b) Let $f' \colon \Phi(X) \to Y'$ be an epimorphism. Let $R' = \Eq(f')$, which is a subobject of $\Phi(X \times X)$. As discussed above, $R'=\Phi(R)$ for a unique $R \subset X \times X$. Let $f \colon X \to Y$ be the quotient of $X$ by $R$, i.e., the co-equalizer of $R \rightrightarrows X$. Since $\cB$ is pre-Galois, we have $R=\Eq(f)$. Since $\Phi$ commutes with finite limits, we have $\Eq(f)=\Eq(\Phi(f))$. Thus $f$ and $f'$ are isomorphic arrows by Proposition~\ref{prop:non-degen-factor}, and in particular, $Y' \cong \Phi(Y)$.

(c) The hypothesis, together with (b), implies that $\Phi$ is essentially surjective. Since $\Phi$ is fully faithful by (a), it is an equivalence.
\end{proof}

\begin{remark}
One can use the lemma to show that any embedding commutes with finite co-limits.
\end{remark}

In what follows, for a group $G$ we write $\bS'(G)$ for the category of all sets equipped with a $G$-action.

\begin{lemma} \label{lem:uncountable-2}
Let $\cB$ be a pre-Galois category, let $G$ be a group, and let $\Phi \colon \cB \to \bS'(G)$ is a functor that commutes with finite limits and co-products, and maps atoms to transitive $G$-sets. Then there is a quotient $\overline{G}$ of $G$ and an admissible topology on $\overline{G}$ such that $\Phi$ induces an equivalence $\cB \to \bS(\overline{G})$.
\end{lemma}

\begin{proof}
Let $N$ be the set of elements of $G$ that act trivially on $\Phi(X)$ for all $X$. This is clearly a normal subgroup of $G$, and we put $\overline{G}=G/N$. The action of $G$ on $\Phi(X)$ factors through $\overline{G}$, and so $\Phi$ can be regarded as valued in $\bS'(\overline{G})$. For notational simplicity, we replace $G$ with $\overline{G}$ in what follows, and thereby reduce to the case $N=1$.

Let $\cU$ be the set of all subgroups of $G$ that occur as stabilizers on $\Phi(X)$, for some $X \in \cB$. The set $\cU$ is closed under finite intersections, since if $x \in \Phi(X)$ and $y \in \Phi(Y)$ then $G_x \cap G_y = G_{(x,y)}$ where $(x,y) \in \Phi(X) \times \Phi(Y) \cong \Phi(X \times Y)$. It therefore forms a neighborhood basis of the origin for a topology on $G$. This topology is non-archimedean by definition, and Hausdorff due to the reduction made in the previous paragraph.

We claim that $G$ is Roelcke pre-compact. Let $U$ and $V$ be given open subgroups of $G$. Then there exist $x \in \Phi(X)$ and $y \in \Phi(Y)$ such that $G_x \subset U$ and $G_y \subset V$. Since $U \backslash G /V$ is a quotient of $G_x \backslash G/G_y$, it suffices to show that the latter is finite. We have
\begin{displaymath}
G_x \backslash G/G_y  \cong G \backslash (G/G_x \times G/G_y) \subset G \backslash (\Phi(X) \times \Phi(Y)) \cong G \backslash \Phi(X \times Y).
\end{displaymath}
Since $\Phi$ commutes with finite co-products and takes atoms to transitive sets, it follows that $\Phi(Z)$ is finitary for any $Z$, and so the above set is finite. This verifies the claim. We have thus shown that $G$ is admissible.

Now, for any $X \in \cB$, the action of $G$ on $\Phi(X)$ is smooth by definition of the topology on $G$. It is also finitary, as noted above. Thus $\Phi$ naturally maps into the category $\bS(G)$. The induced functor $\Phi' \colon \cB \to \bS(G)$ is clearly an embedding of pre-Galois categories. It is an equivalence by Lemma~\ref{lem:uncountable-2}. Indeed, if $U$ is any open subgroup of $G$ then there exists $x \in \Phi(X)$, for some atom $X$, such that $G_x \subset U$, and so $G/U$ is a quotient of $G/G_x \cong \Phi(X)$. This shows that every transitive $G$-set is a quotient of some $\Phi(X)$, from which it follows that every object of $\bS(G)$ is such a quotient.
\end{proof}

\begin{lemma} \label{lem:uncountable-3}
Let $\{\cB_i\}_{i \in I}$ be a directed system of pre-Galois categories, where the transition maps are embeddings, and let $\cB$ be the 2-colimit. Suppose that for each $i$ there is an admissible group $G_i$ such that $\cB_i \cong \bS(G_i)$. Then there is an admissible group $G$ such that $\cB \cong \bS(G)$.
\end{lemma}

\begin{proof}
We may as well suppose $\cB_i=\bS(G_i)$ for all $i$, which we do for simplicity. For $i \le j$, we let $\Psi_{i,j} \colon \cB_i \to \cB_j$ be the transition functor. We note that $\cB$ is a pre-Galois category; this is easy to see directly from the definitions.

Say that a subset $J$ of $I$ is \defn{big} if there is some $i \in I$ such that $J$ contains every $j$ with $j \ge i$. The big subsets of $I$ form a filter, and so there is some ultrafilter $\cF$ on $I$ that contains the big sets. Let $G^*$ be the ultraproduct of the $G_i$'s with respect to $\cF$.

We define a functor $\Phi \colon \cB \to \bS'(G^*)$. To this end, we first define a functor $\Phi_i \colon \cB_i \to \bS'(G^*)$. Let $X$ be a $G_i$-set. Define $\Phi_i(X)$ to be the ultraproduct of the sets $\{ \Psi_{i,j}(X) \}_{j \in I}$, where here we let $\Psi_{i,j}(X)$ be a one-point set if $i \nlet j$. Since the $j$th set in the ultraproduct carries an action of $G_j$, it follows that $\Phi_i(X)$ carries an action of $G^*$. It is clear that the construction is functorial. Now suppose $i \le j$, and let $X$ be a $G_i$-set. Then the factors in the ultraproducts defining $\Phi_i(X)$ and $\Phi_j(\Psi_{i,j}(X))$ are naturally isomorphic on a big set (namely the indices $\ge j$), and so the ultraproducts themselves are naturally isomorphic. This shows that the $\Phi_i$'s define a functor $\Phi$ on the 2-colimit.

Ultraproducts commute with finite products and co-products, and fiber products. Thus $\Phi$ commutes with such constructions as well. If $X$ is a transitive $G_i$-set then $\Psi_{i,j}(X)$ is a transitive $G_j$-set for all $j$, and so $\Phi_i(X)$ is a transitive $G^*$-set. Thus $\Phi$ maps atoms of $\cB$ to transitive $G^*$-sets. It follows from Lemma~\ref{lem:uncountable-2} that there is a quotient $G$ of $G^*$ and an admissible topology on $G$ such that $\Phi$ induces an equivalence $\cB \cong \bS(G)$.
\end{proof}

\begin{proof}[Proof of Theorem~\ref{thm:uncountable}]
Let $\cB$ be a given pre-Galois category. For a finite set $S$ of objects of $\cB$, let $\cB_S$ be the sub pre-Galois category generated by $S$; this is the full subcategory of $\cB$ spanned by objects which occur as a subquotient of sums and products of objects in $S$. It follows from Corollary~\ref{cor:finite-quot} that $\cB_S$ contains countably many isomorphism classes. Thus $\cB_S \cong \bS(G_S)$ for some admissible group $G_S$ by Theorem~\ref{thm:fraisseB}.

Now, let $\Sigma$ be a set of objects of $\cB$ such that every object is isomorphic to some object of $\Sigma$. This exists since $\cB$ is essentially small, by assumption. Then $\cB$ is the 2-colimit of the system $\{\cB_S\}_{S \subset \Sigma}$. As the transition maps in this system are embeddings, it follows from Lemma~\ref{lem:uncountable-3} that $\cB \cong \bS(G)$ for some admissible group $G$.
\end{proof}

\section{Examples from relational structures} \label{s:examples}

We now look at some examples of $\rA$-categories and $\rB$-categories coming from classes of relational structures. See \S \ref{ss:fraisse-class} for basic definitions on relational structures.

\subsection{General comments}

Let $\cC$ be a non-empty class of finite relational structures. We assume throughout this section that $\cC$ is hereditary and that $\vert \cC_n \vert$ is finite for all $n \ge 0$. Recall that we can regard $\cC$ as a category, with morphisms being embeddings of structures.

\begin{proposition} \label{prop:A-struct}
The category $\cC$ is an $\rA_1$-category, and the following are equivalent:
\begin{enumerate}
\item $\cC$ is an $\rA$-category.
\item Given $Y \in \cC$ and a proper substructure $X \subset Y$, there exists a structure $Z \in \cC$ and distinct embeddings $Y \rightrightarrows Z$ that have equal restriction to $X$.
\item Given $Y \in \cC$ and a proper substructure $X \subset Y$, there exists a non-trivial self-amalgamation of $Y$ over $X$.
\end{enumerate}
\end{proposition}

\begin{proof}
The class $\cC$ contains the empty structure since it is non-empty and hereditary. It is clear that the empty structure is the initial object of $\cC$, and so $\cC$ has an initial set. This verifies Definition~\ref{defn:Acat}(a).

Let $(\beta, \gamma)$ be a pre-amalgamation, where $\beta \colon A \to B$ and $\gamma \colon A \to C$. Consider an object $(\delta, \epsilon)$ of $\Amalg(\beta,\gamma)$, where $\delta \colon B \to D$ and $\epsilon \colon C \to D$. We say that $(\delta,\epsilon)$ is \defn{minimal} if $\delta$ and $\epsilon$ are jointly surjective, i.e., $D=\im(\delta) \cup \im(\epsilon)$. Every object of $\Amalg(\beta,\gamma)$ admits a unique (up to isomorphism) map from a minimal object. Indeed, in the above notation, let $D'=\im(\delta) \cup \im(\epsilon)$, regarded as a substructure of $D$. Then $D'$ is a minimal, with structure maps $\delta$ and $\epsilon$, and the inclusion $D' \to D$ is a map in $\Amalg(\beta,\gamma)$; a key point here is that $D'$ still belongs ot the class $\cC$ since $\cC$ is hereditary.

Let $S$ be a set of isomorphism class representatives for the minimal objects of $\Amalg(\beta,\gamma)$. The above argument shows that $S$ is an amalgamation set for $(\beta, \gamma)$. Since the cardinality of a minimal object is at most $\# B + \# C$ and we have assumed $\vert \cC_n \vert$ is finite for all $n$, it follows that $S$ is finite. This verifies Definition~\ref{defn:Acat}(b).

We have already explained (at the end of \S \ref{ss:Acat}) how the remaining three conditions are equivalent.
\end{proof}

Suppose that $\cC$ is indeed an $\rA$-category and that it is also satisfies the amalgamation property; then $\cC$ is a Fra\"iss\'e class. Let $\Omega$ be the Fra\"iss\'e limit, and let $G=\Aut(\Omega)$, which acts oligomorphically on $\Omega$. Theorem~\ref{thm:fraisseA-1} gives an equivalence of $\cA$ with $\bT(G;\sE)^{\op}$, where $\sE$ is the set of subgroups of $G$ of the form $G(A)$ where $A \subset \Omega$ is a finite subset. (Recall that $G(A)$ is the subgroup of $G$ fixing each element of $A$.)

\subsection{Sets} \label{ss:sets}

Let $\cC$ be the class of all finite sets (the signature in this case is empty). This is an $\rA$-category by Proposition~\ref{prop:A-struct}. The amalgamation property holds. The Fra\"iss\'e limit is the countable set $\Omega=\{1,2,\ldots\}$ and its automorphism group is the infinite symmetric group $\fS$. Let $\sE$ be the stabilizer class consisting conjugates of $\fS(n)$, for variable $n$ (see Example~\ref{ex:sym-stab-class}). Then we have an equivalence of categories $\cC \cong \bT(\fS;\sE)^{\op}$.

We can also describe the $\rA$-category $\bT(\fS)^{\op}$. Define a category $\cC'$ as follows. An object is a pair $(X,G)$ where $X$ is a finite set and $G$ is a subgroup of the symmetric group $\Perm(X)$ on $X$. A morphism $(X,G) \to (Y,H)$ is an injective function $\alpha \colon X \to Y$ such that $H$ is contained in $G$, where here we identify $\Perm(X)$ with $\Perm(\im(\alpha))$, which we in turn regard as a subgroup of $\Perm(Y)$ in the usual manner. Then $\bT(\fS)^{\op}$ is equivalent to $\cC'$.

%

\subsection{Total orders} \label{ss:tot-ord}

Let $\cC$ be the class of finite totally ordered sets (the signature consists of a single binary relation). This is an $\rA$-category by Proposition~\ref{prop:A-struct}, and the amalgamation property holds. The Fra\"iss\'e limit $\Omega$ is the set of rational numbers, with its usual order. Let $G=\Aut(\Omega)$. It turns out that every open subgroup of $G$ has the form $G(A)$ for some finite subset $A \subset \Omega$ \cite[Proposition~17.1]{repst}. We thus have an equivalence $\cC \cong \bT(G)^{\op}$. The Delannoy category studied in \cite{line} is associated to this group $G$.

\subsection{The countable matching} \label{ss:matching}

Let $\cC$ be the class of all simple graphs in which each vertex belongs to at most one edge; the signature consists of a single binary relation (the edge relation on vertices). This is a Fra\"iss\'e class. The limit $\Omega$ is a perfect matching on a countable vertex set. Its automorphism group $G$ is the wreath product $\bZ/2 \wr \fS$, where $\fS$ is the infinite symmetric group.

The category $\cC$ is an $\rA_1$-category by Proposition~\ref{prop:A-struct}, but it is \emph{not} an $\rA$-category. To see this, let $Y$ be a single edge, and let $X \subset Y$ be one of the vertices. Then any map $X \to Z$ admits at most one extension to $Y$, and so Proposition~\ref{prop:A-struct}(b) fails. Alternatively, the only self-amalgamation of $Y$ over $X$ is the trivial one, and so Proposition~\ref{prop:A-struct}(c) fails.

Theorem~\ref{thm:fraisseA-2} does produce a faithful and essentially surjective functor $\Phi \colon \cC \to \bT(G; \sE)^{\op}$, for an appropriate stabilizer class $\sE$. We can see directly that this functor is not full: indeed, the map $\Phi(Y) \to \Phi(X)$ is an isomorphism since every embedding $X \to \Omega$ extends uniquely to $Y$. The inverse map does not come from a map $Y \to X$ in $\cC$, as there are no such maps.

Let $\cC_0$ be the (non-hereditary) subclass of $\cC$ consisting of graphs in which each vertex belongs to exactly one edge. Then $\Phi$ restricts to an equivalence $\cC_0 \to \bT(G; \sE)^{\op}$.

\subsection{Permutation classes}

Let $\cP$ be the class of all finite sets equipped with a pair of total orders. Let $X$ be a structure of $\cP$. Label the elements of $X$ as $1, 2, \ldots, n$ according to the first order. We can then enumerate the elements of $X$ under the second order to get a string in the alphabet $\{1,\ldots,n\}$ in which each letter appears once. This string exactly determines the isomorphism type of $X$. We can thus view structures in $\cP$ as permutations, and thus typically use symbols like $\sigma$ for its members. The embedding order on $\cP$ is the so-called containment order on partitions.

A \defn{permutation class} is a non-empty hereditary subclass $\cC$ of $\cP$. There is an extensive literature on permutation classes; for an overview, see \cite{Vatter}. We mention one relevant result here: a theorem of Cameron \cite{CameronPerm} asserts that there are exactly five permutation classes that are Fra\"iss\'e classes.

Let $\sigma$ be a permutation of length $n$, and let $\alpha_1, \ldots, \alpha_n$ be other permutations of lengths $m_1, \ldots, m_n$. There is then a permutation $\sigma[\alpha_1, \ldots, \alpha_n]$ of length $m=m_1+\cdots+m_n$, called \defn{inflation}. We refer to \cite[\S 3.2]{Vatter} for the definition, and just give an example here:
\begin{displaymath}
231[12,321,3412] = 56\ 987\ 3412.
\end{displaymath}
We have inserted spaces into the result to make the operation more clear. The three components on the right correspond to the three permutations in the brackets. Each uses an interval of numbers, and the order of the intervals is determined by the outside permutation. A permutation class $\cC$ is \defn{substitution closed} if $\sigma[\alpha_1, \ldots, \alpha_n]$ belongs to $\cC$ whenever $\sigma,\alpha_1, \ldots, \alpha_n$ all belong to $\cC$.

\begin{proposition}
Let $\cC$ be a substitution closed permutation class containing some permutation of length $\ge 2$. Then $\cC$ is an $\rA$-category.
\end{proposition}

\begin{proof}
Let $\tau \to \sigma$ be a non-isomorphism in $\cC$, and suppose the embedding misses $i \in \sigma$. Let $n$ be the length of $\sigma$, and consider the inflation $\sigma'=\sigma[\alpha_1, \ldots, \alpha_n]$ where $\alpha_j=1$ for $j \ne i$, and $\alpha_i$ has length length~2. Note that $\cC$ contains the permutation 1 and some permutation of length~2 since it is hereditary. One easily sees that $\sigma'$ is a non-trivial self-amalgamation of $\sigma$ over $\tau$. Thus $\cC$ is an $\rA$-category by Proposition~\ref{prop:A-struct}.
\end{proof}

\begin{example} \label{ex:separable}
A permutation is \defn{separable} if it can be built from the permutation 1 with sums and skew-sums; the empty permutation is also separable. (Given two permutations $\alpha$ and $\beta$ their sum is $12[\alpha,\beta]$ and their skew-sum is $21[\alpha,\beta]$.) Equivalently, a permutation is separable if the permutations 2413 and 3142 do not embed into it. The class $\cC$ of all separable permutations is a substitution closed permutation class. It is thus an $\rA$-category by the above proposition.

The class $\cC$ does not have the amalgamation property. To see this, regard 123 as a subpermutation of 1342 (using the first three positions) and 3124 (using the last three positions). In the class of all permutations, there is a unique amalgamation, namely 41352. This is not separable, since when the middle 3 is deleted we obtain 3142. However, $\cC$ does have the joint embedding property, which means that any two objects embed into a common third object: indeed, if $\alpha$ and $\beta$ are separable permutations then $\alpha$ and $\beta$ each embed into their sum $12[\alpha,\beta]$, which is also separable.

Let $\cB=\BB(\cC)$. Then $\cB$ is a $\rB$-category. Since $\cC$ has an initial object, the final object $\bone$ of $\cB$ is atomic. Since the amalgamation property fails for $\cC$, it follows that there are maps of atoms $X \to Z$ and $Y \to Z$ in $\cB$ such that $X \times_Z Y=\bzero$ (indeed, take $X$, $Y$, and $Z$ to be the atoms corresponding to the permutations 1342, 3124, and 123 discussed above). However, since $\cC$ has the joint embedding property, it follows that $X \times Y$ is non-empty for all atoms $X$ and $Y$ of $\cB$.
\end{example}

\end{document}